\documentclass[11pt]{amsart}

 \usepackage{amsfonts,graphics,amsmath,amsthm,amsfonts,amscd,amssymb,amsmath,latexsym,multicol,
 mathrsfs}
\usepackage{epsfig,url}
\usepackage{flafter}
\usepackage{fancyhdr}
\usepackage{hyperref}
\hypersetup{colorlinks=true, linkcolor=black}
\usepackage{xcolor}

 \usepackage[matrix, arrow]{xy}

\DeclareMathOperator{\Div}{Div}

\DeclareMathOperator{\Supp}{Supp}

\DeclareMathOperator{\vol}{vol}

 \numberwithin{equation}{subsection}
 \numberwithin{footnote}{subsection}

 \newtheorem{cor}[subsection]{Corollary}
 \newtheorem{lem}[subsection]{Lemma}
 \newtheorem{prop}[subsection]{Proposition}
 \newtheorem{thm}[subsection]{Theorem}
 \newtheorem{conj}[subsection]{Conjecture}

{
\theoremstyle{upright}

 \newtheorem{exa}[subsection]{Example}
 
 \newtheorem{rem}[subsection]{Remark}

}

 \newcommand{\N}{\mathbb N}
 \newcommand{\PP}{\mathbb P}
 
 \newcommand{\Q}{\mathbb Q}
 \newcommand{\R}{\mathbb R}
 \newcommand{\Z}{\mathbb Z}
  \newcommand{\C}{\mathbb C}
 \newcommand{\bir}{\dashrightarrow}
 \newcommand{\rddown}[1]{\left\lfloor{#1}\right\rfloor} 

\title{\large G\MakeLowercase{eometry of polarised varieties}}
\author{\large C\MakeLowercase{aucher} B\MakeLowercase{irkar}}
\date{\today}
\begin{document}
\maketitle


\begin{abstract}
In this paper, we investigate the geometry of projective varieties polarised by ample and more generally nef and big 
Weil divisors. First we study birational boundedness of linear systems. We show that 
if $X$ is a projective variety of dimension $d$ with $\epsilon$-lc singularities for
$\epsilon>0$, and if $N$ is a nef and big Weil divisor on $X$ such that $N-K_X$ is pseudo-effective, 
then the linear system $|mN|$ defines a birational map for some natural number $m$ depending only on $d,\epsilon$.  
This is key to proving various other results. 
For example, it implies that if $N$ is a big Weil divisor (not necessarily nef) on a klt Calabi-Yau variety of dimension $d$, 
then the linear system $|mN|$ defines a birational map for some natural number $m$ depending only 
on $d$. It also gives new proofs of some known results, for example, if 
$X$ is an $\epsilon$-lc Fano variety of dimension $d$ then taking 
$N=-K_X$  we recover birationality of $|-mK_X|$ for bounded $m$. 

We prove similar birational boundedness 
results for nef and big Weil divisors $N$ on projective klt varieties $X$ when both $K_X$ and $N-K_X$ are pseudo-effective 
(here $X$ is not assumed $\epsilon$-lc).

Using the above, we show boundedness of polarised varieties under some natural conditions. 
We extend these to boundedness of semi-log canonical Calabi-Yau pairs polarised by 
effective ample Weil divisors not containing lc centres. We will briefly discuss applications to existence of projective coarse moduli spaces of such polarised Calabi-Yau pairs.
 
\end{abstract}

\tableofcontents


\section{\bf Introduction}

We work over an algebraically closed field $k$ of characteristic zero unless stated otherwise. By 
\emph{integral divisor} we will mean a Weil divisor with integer coefficients which is not necessarily Cartier.\\ 

Assume that $X$ is a normal projective variety and $N$ is an  
integral divisor on $X$. For each natural number $m$ we have the linear system $|mN|$. 
In algebraic geometry it is often a central theme 
to understand such linear systems and their associated maps $\phi_{|mN|}\colon X \bir \PP^{h-1}$ defined by 
a basis of $H^0(mN)$ where $h=h^0(mN)$. When $N$ is ample (or just nef and big)
a main problem is to estimate those $m$ for which the linear system $|mN|$ defines a birational map, 
ideally for $m$ bounded in terms of some basic invariants of $X$. 
We need to impose 
some conditions to get reasonable results and in practice this means we should somehow  
get the canonical divisor of $X$ involved. 

Indeed there has been extensive studies in the literature when $N=K_X$ is ample or $N=-K_X$ is ample. 
If $X$ has log canonical (lc) singularities and $N=K_X$ is ample, then $|mN|$ defines a birational map 
for some $m$ depending on $\dim X$  [\ref{HMX2}]. On the other hand, 
when $X$ has $\epsilon$-log canonical ($\epsilon$-lc) singularities 
with $\epsilon>0$ and $N=-K_X$ is ample, then $|mN|$ defines a birational map for some $m$ depending on $\dim X, \epsilon$ 
 [\ref{B-compl}, theorem 1.2]; here the $\epsilon$-lc condition cannot be removed.

In this paper, we study the linear systems $|mN|$ in a rather general context 
when $N$ is nef and big (we will also consider the non-nef and non-big cases in some places). 
As before we would like to see when there is an $m$ depending only on dimension of $X$ and 
some other data so that $|mN|$ defines a birational map. It turns out that if $X$ has 
$\epsilon$-lc singularities with $\epsilon>0$ and if $N-K_X$ is pseudo-effective, 
then $|mN|$ defines a birational map for some 
$m$ depending only on $\dim X,\epsilon$ (see Theorem \ref{t-eff-bir-nefbig-weil}). 
The result in particular can be applied to varieties with terminal and canonical singularities.

On the other hand, we show that if $X$ has klt singularities and if both $K_X$ and $N-K_X$ are pseudo-effective, 
then $|mN|$ defines a birational map for some 
$m$ depending only on $\dim X$ (see Theorem \ref{t-eff-bir-nefbig-klt}). A corollary of each of these results is that 
when $X$ is klt Calabi-Yau, i.e. $X$ is projective with klt singularities and $K_X\equiv 0$, then $|mN|$ defines a birational map 
for some $m$ depending only on $\dim X$; in this case we do not even need to assume $N$ to be nef but only big 
(see Corollary \ref{cor-cy}).

Applying the results of the previous paragraph we prove boundedness of varieties under certain conditions. 
Let $d$ be a natural number and $\epsilon,v$ be positive rational numbers. 
If $X$ is a projective variety of dimension $d$ with $\epsilon$-lc singularities, $K_X$ is nef, and $N$ is a 
nef and big integral divisor with volume $\vol(K_X+N)\le v$, then $X$ belongs to a bounded family 
(see Theorem \ref{t-bnd-nef-pol-pairs}). In particular, if $X$ is a klt Calabi-Yau variety 
of dimension $d$ and $N$ is a nef and big integral divisor with $\vol(N)\le v$, then $X$ 
belongs to a bounded family (see Corollary \ref{cor-bnd-cy}). 

In the Calabi-Yau case we can further prove boundedness in the semi-log canonical (slc) case. 
Slc schemes are higher dimensional analogues of nodal curves which may not be normal nor irreducible.
If $X$ is an slc Calabi-Yau of dimension $d$ and $N\ge 0$ is an ample integral divisor such that 
$(X,uN)$ is slc for some $u>0$ and if $\vol(N)=v$, then $X$ belongs to a bounded family 
(see Corollary \ref{cor-bnd-semi-lc-pol-CY}). Such $X$ are called polarised Calabi-Yau. 
Similar boundedness holds for slc Calabi-Yau pairs $(X,B)$.

The boundedness results just mentioned provide an important ingredient for constructing projective moduli spaces. In general, Calabi-Yau varieties do not carry any ``canonical" polarisation. To form moduli spaces 
one needs to take ample divisors with certain properties, e.g. fixed volume and bounded Cartier index. 
However, to get a compact moduli space one needs to consider limits of such polarised 
varieties and this causes problems. One issue is that the limiting space may not be normal 
any more, that is, one has to consider slc schemes. Another problem is that we need these limiting spaces to be bounded 
in order to get a finite type moduli space. The boundedness mentioned in the last 
paragraph is exactly what we need. 

In the rest of this introduction we will state the results mentioned above in more general forms. 
We actually prove even more general versions of many of them later in the paper.\\

{\textbf{\sffamily{Birational boundedness for nef and big integral divisors.}}}
The first main result of this paper is the following.

\begin{thm}\label{t-eff-bir-nefbig-weil}
Let $d$ be a natural number and $\epsilon$ be a positive real number. 
Then there exists a natural number $m$ depending only on $d,\epsilon$ satisfying the following. Assume that 
\begin{itemize}
\item $X$ is a projective $\epsilon$-lc variety of dimension $d$, 

\item $N$ is a nef and big integral divisor on $X$, and 

\item $N-K_X$ is pseudo-effective.
\end{itemize}
Then $|m'N+L|$ and $|K_X+m'N+L|$ define birational maps for any natural number $m'\ge m$ and any
integral pseudo-effective divisor $L$. 
\end{thm}

We actually prove a more general statement in which we replace the assumption of $N$ being integral with assuming 
$N=E+R$ where $E$ is integral and pseudo-effective and $R\ge 0$ is an $\R$-divisor whose non-zero coefficients 
are $\ge \delta$ for some fixed $\delta>0$ (see Theorem \ref{t-eff-bir-general}). Similarly we will prove 
more general forms of many of the results below.  

Note that if $-K_X$ is pseudo-effective, then $N-K_X$ is automatically pseudo-effective. 
This is in particular useful on Calabi-Yau pairs as we will see later. Also note that 
instead of $X$ being $\epsilon$-lc we can assume $(X,B)$ is $\epsilon$-lc for some boundary 
$B$ because we can apply the theorem on a $\Q$-factorialisation of $X$.

The theorem in particular applies well to the following three cases:
\begin{enumerate}
\item when $K_X$ is nef and big and $N=K_X$,
\item when $-K_X$ is nef and big and $N=-K_X$, and 
\item when $K_X\equiv 0$ and $N$ is nef and big.
\end{enumerate}

Cases (1) and (2) are well-known by Hacon-M$\rm ^c$Kernan-Xu [\ref{HMX2}] 
(see also [\ref{HM-bir-bnd}][\ref{Ta}][\ref{Tsuji}]) and Birkar [\ref{B-compl}], respectively; 
however, we reprove these results as we only rely on some  
of the ideas and constructions of [\ref{HMX2}] and [\ref{B-compl}]. In case (1) our 
proof is essentially the same as the proof in [\ref{HMX2}]. But in case (2) we get a new proof  
 which is in some sense quite different from the proof in [\ref{B-compl}] 
despite similarities of the two proofs because here we do not use boundedness of complements 
in dimension $d$ (see \ref{rem-Fano-eff-bir-new-proof} for more details); but this relies on the BAB 
[\ref{B-BAB}, Theorem 1.1] in lower dimension which is reasonable as we want to apply induction on dimension.
 Case (3) is new which we will state below more precisely in \ref{cor-cy}.

\begin{cor}\label{cor-eff-bir-nefbig-weil}
Let $d$ be a natural number and $\epsilon$ be a positive real number. 
Then there exist natural numbers $m,l$ depending only on $d,\epsilon$ satisfying the following. Assume that 
\begin{itemize}
\item $X$ is a projective $\epsilon$-lc variety of dimension $d$, and 

\item $N$ is a nef and big integral divisor on $X$.
\end{itemize}
\vspace{0.2cm}
Then $|m'K_X+l'N+L|$ defines a birational map for any natural numbers $m'\ge m$ and $l'\ge lm'$ and any
pseudo-effective integral divisor $L$.
 
\end{cor}

In the above results we cannot drop the $\epsilon$-lc assumption, see Example \ref{exa-3}. 
However, we can replace it with some other conditions as in the next result.

\begin{thm}\label{t-eff-bir-nefbig-klt}
Let $d$ be a natural number and $\Phi\subset [0,1]$ be a DCC set of rational numbers. 
Then there is a natural number $m$ depending only on $d,\Phi$ satisfying the following.
Assume 
\begin{itemize}
\item $(X,B)$ is a klt projective pair of dimension $d$, 

\item the coefficients of $B$ are in $\Phi$, 

\item $N$ is a nef and big integral divisor, and

\item  $N-(K_X+B)$ and $K_X+B$ are pseudo-effective. 
\end{itemize}
Then $|m'N+L|$ and $|K_X+m'N+L|$ define birational maps for any natural number $m'\ge m$ and
any integral pseudo-effective divisor $L$.
\end{thm}  

The theorem does not hold if we replace the klt property of $(X,B)$ with lc. Indeed 
any klt Fano variety $X$ of dimension $d$ admits an lc $n$-complement $K_X+B$ 
for some $n$ depending only on $d$ [\ref{B-compl}] (so $n(K_X+B)\sim 0$)
but taking $N=-K_X$ there is no bounded $m$ so that $|mN|$ defines a birational map 
as Example \ref{exa-3} shows.\\

{\textbf{\sffamily{Birational boundedness for big integral divisors on Calabi-Yau pairs.}}}
A consequence of both \ref{t-eff-bir-nefbig-weil} and \ref{t-eff-bir-nefbig-klt} 
is a birational boundedness statement regarding 
Calabi-Yau pairs. A Calabi-Yau pair  is a projective pair $(X,B)$ with $K_X+B\sim_\R 0$; 
we do not assume vanishing of $h^i(\mathcal{O}_X)$ for $0<i<\dim X$ as is customary in some other contexts.

\begin{cor}\label{cor-cy}
Let $d$ be a natural number and $\Phi\subset [0,1]$ be a DCC set of real numbers. 
Then there is a natural number $m$ depending only on $d,\Phi$ 
satisfying the following. Assume 
\begin{itemize}
\item $(X,B)$ is a klt Calabi-Yau pair of dimension $d$, 

\item the coefficients of $B$ are in $\Phi$, and  

\item $N$ is a big integral divisor on $X$.

\end{itemize}
\vspace{0.2cm}
Then $|m'N+L|$ and $|K_X+m'N+L|$ define birational maps for any natural number $m'\ge m$ and
 any integral pseudo-effective divisor $L$. 
In particular, the volume $\vol(N)\ge \frac{1}{m^d}$.
\end{cor}

A yet special case of this is when $B=0$, say when $\Phi=\{0\}$, in which case 
$m$ depends only on $d$. 

Note that the corollary only assumes klt singularities 
rather than $\epsilon$-lc. In fact, we will see that such $X$ automatically have $\epsilon$-lc singularities 
for some $\epsilon>0$ depending only on $d$ (this follows from [\ref{HMX2}]), 
so we can apply Theorem \ref{t-eff-bir-nefbig-weil} immediately after taking a minimal model of $N$.

The corollary was proved by Jiang [\ref{Jiang-CY}] in dimension 3 when $X$ is a Calabi-Yau  
with terminal singularities. His proof is entirely different as 
it relies on the Riemann-Roch theorem for 3-folds with terminal singularities. 
More special cases for 3-folds were obtained earlier by Fukuda [\ref{Fukuda}] and 
Oguiso-Peternell [\ref{OP}]. The smooth surface case goes back to Reider [\ref{Reider}]. 
Also see [\ref{KMPP}] for recent relevant results on irreducible symplectic varieties.  
\\

{\textbf{\sffamily{Boundedness of polarised $\epsilon$-log canonical nef pairs.}}}
Given a projective variety $X$ (or more generally pair) polarised by a nef and big integral divisor $N$, 
we would like to find conditions 
which guarantee that $X$ belongs to a bounded family. This is often achieved by controlling positivity and 
singularities. For example, if $X$ is $\epsilon$-lc and $N=-K_X$ is nef and big, then $X$ is bounded [\ref{B-BAB}]. 
On the other hand, if $X$ is $\epsilon$-lc and $N=K_X$ is ample with volume bounded from above, then 
$X$ is bounded (this follows from the results of [\ref{HMX2}]). 
The next result deals with the case when $K_X$ (and more generally $K_X+B$) is nef.

\begin{thm}\label{t-bnd-nef-pol-pairs}
Let $d$ be a natural number and $\epsilon,\delta,v$ be positive real numbers. 
Consider pairs $(X,B)$ and divisors $N$ on $X$ such that  
\begin{itemize}
\item $(X,B)$ is projective $\epsilon$-lc of dimension $d$,

\item the coefficients of $B$ are in $\{0\}\cup [\delta,\infty)$,

\item $K_X+B$ is nef, 

\item $N$ is nef and big and integral, and

\item $\vol(K_X+B+N)\le v$.
\end{itemize}
Then the set of such $(X,\Supp B)$ forms a bounded family. 
If in addition $N\ge 0$, then the set of such $(X,\Supp(B+N))$ forms a bounded family.
\end{thm}

A consequence of this is the following.

\begin{cor}\label{cor-bnd-cy}
Let $d$ be a natural number, $v$ be a positive real number, and $\Phi\subset [0,1]$ be a DCC set of real numbers. 
Consider pairs $(X,B)$ and divisors $N$ on $X$ satisfying the following: 
\begin{itemize}
\item $(X,B)$ is a klt Calabi-Yau pair of dimension $d$,

\item the coefficients of $B$ are in $\Phi$, and 

\item $N$ is nef and big and integral, and 

\item $\vol(N)\le v$.

\end{itemize}
Then the set of such $(X, \Supp B)$ forms a bounded family.
If in addition $N\ge 0$, then the set of such $(X,\Supp(B+N))$ forms a bounded family.
\end{cor}

The point is that $(X,B)$ is automatically $\epsilon$-lc for some $\epsilon>0$ depending only on $d,\Phi$, 
so we can apply Theorem \ref{t-bnd-nef-pol-pairs}. Note that if we relax the nef and big property 
of $N$ to only big, then  $X$ is birationally bounded as we can apply the 
corollary to the minimal model of $N$. The special case of the corollary in which $N$ is Cartier and $B=0$ 
was proved earlier in [\ref{MST}, Corollary 10].

The corollary is crucially used  in Odaka [\ref{Odaka}] to get partial compactification of moduli spaces of Calabi-Yau varieties polarised by ample line bundles (see also [\ref{Odaka-2}]).\\

{\textbf{\sffamily{Boundedness of polarised semi-log canonical Calabi-Yau pairs.}}}
The above boundedness statements are not enough for construction of moduli spaces. 
The problem is that limits of families of $\epsilon$-lc varieties are not necessarily 
$\epsilon$-lc. In fact the limit may not even be irreducible. 
We want to address this problem by showing boundedness 
of appropriate classes of Calabi-Yau pairs. 
An \emph{slc Calabi-Yau pair} is a projective slc pair $(X,B)$ such that $K_X+B\sim_\R 0$. 
The desired boundedness is a consequence of the next result on lc thresholds.

\begin{thm}\label{t-lct-nef-big-slc-cy}
Let $d$ be a natural number, $v$ be a positive real number, and $\Phi\subset [0,1]$ be a DCC set of real numbers. 
Then there is a positive real number $t$ depending only on $d,v,\Phi$ satisfying the following.
Assume that   
\begin{itemize}
\item $(X,B)$ is an slc Calabi-Yau pair of dimension $d$, 

\item the coefficients of $B$ are in $\Phi$,

\item $N\ge 0$ is a nef integral divisor on $X$,

\item $(X,B+uN)$ is slc for some real number $u>0$, and 

\item for each irreducible component $S$ of $X$, $N|_S$ is big with $\vol(N|_S)\le v$.
\end{itemize}
Then $(X,B+tN)$ is slc. 
\end{thm}

The key point is that $t$ does not depend on $u$. 

A \emph{polarised slc Calabi-Yau pair} consists of a connected slc Calabi-Yau pair 
$(X,B)$ and an ample integral divisor $N\ge 0$ such that $(X,B+uN)$ is  slc for some real number $u>0$. 
We refer to such a pair by saying $(X,B),N$ is a polarised slc Calabi-Yau pair.

\begin{cor}\label{cor-bnd-semi-lc-pol-CY}
Let $d$ be a natural number, $v$ be a positive real number, and $\Phi\subset [0,1]$ be a DCC set of rational numbers. 
Consider $(X,B)$ and $N$ such that 
\begin{itemize}
\item $(X,B),N$ is a polarised slc Calabi-Yau pair of dimension $d$,

\item the coefficients of $B$ are in $\Phi$, and 

\item $\vol(N)=v$.
\end{itemize}
Then the set of such $(X,\Supp(B+N))$ forms a bounded family.
\end{cor}

This is a consequence of Theorem \ref{t-lct-nef-big-slc-cy} and the main result of 
Hacon-M$\rm ^c$Kernan-Xu [\ref{HMX3}] as we can pick a rational number $t>0$ depending 
only on $d,v,\Phi$ so that $(X,B+tN)$ is a stable pair with 
$$
\vol(K_X+B+tN)=t^dv.
$$

\bigskip

{\textbf{\sffamily{Moduli of polarised slc Calabi-Yau pairs.}}}
Combining the above boundedness results (\ref{t-lct-nef-big-slc-cy}, \ref{cor-bnd-semi-lc-pol-CY}) 
with the moduli theory of stable pairs [\ref{kollar-moduli-gen-type}] implies existence of  
projective coarse moduli spaces for polarised slc Calabi-Yau pairs, fixing appropriate invariants. 
Given a natural number $d$ and positive rational numbers $c,v$, one considers those polarised slc Calabi-Yau pairs $(X,B),N$ such that $\dim X=d$, $B=cD$ where $D$ is an integral divisor, and $\vol(N)=v$. 
One needs to consider a slightly more general setting where $N$ may not be integral but $N=cE$ for an integral divisor $E$.
Then such pairs admit a projective coarse moduli space.
The precise statement and its proof will be discussed in a forthcoming work on moduli of varieties.

\bigskip


{{\textbf{\sffamily{Plan of the paper.}}}}
In Section 2 we collect some preliminary definitions and results. 
In Section 3 we study non-klt centres and adjuction on such centres in depth establishing 
results that are key to the subsequent sections. In Section 4 we treat birational 
boundedness of linear systems on $\epsilon$-lc varieties which will form the basis 
for subsequent sections, in particular, we prove more general forms of 
\ref{t-eff-bir-nefbig-weil} and \ref{cor-eff-bir-nefbig-weil} together with \ref{cor-cy}. 
In Section 5 we treat birational boundedness on  
pseudo-effective pairs proving a more general form of \ref{t-eff-bir-nefbig-klt}. 
In Section 6 we establish boundedness results including more general forms of 
\ref{t-bnd-nef-pol-pairs} and \ref{t-lct-nef-big-slc-cy} together with \ref{cor-bnd-semi-lc-pol-CY}.  
Finally in Section 7 we present some examples, remarks, and conjectures. 
\\

{{\textbf{\sffamily{Sketch of some proofs.}}}} 
We sketch proofs of some selected results. 
We start with \ref{t-eff-bir-nefbig-weil}. 
Given a polarised variety $X,N$, the overall strategy for proving birational boundedness is to pick $m$ so that $mN$ has   
sufficiently large volume which ensures that we can create covering families of 
non-klt centres; then one tries to cut down these centres in order to decrease their dimension 
until they are isolated points; 
then one uses vanishing theorems to lift sections from these points in a way that general points 
can be separated hence $|mN|$ defines a birational map. In the end we need to bound the number $m$.

Variants of this strategy have been used numerous times in the literature yielding many fundamental results. 
It originates in the study of the Fujita conjecture by various people and developed over the years 
by many people. Although this is a powerful strategy, 
however, applying it is by no means straightforward. The main difficulty is in decreasing the dimension of 
the non-klt centres. For this we need to somehow apply induction to understand these centres 
and this depends on the setting we start with. 
For example, say if $X$ is of general type (more precisely, its resolution is of general type) 
and $N=K_X$, then the non-klt centres are automatically 
of general type and this is enough to do the induction. But if $X$ is Fano and $N=-K_X$, then 
the non-klt centres may or may not be Fano and this causes problems but the situation can still be managed using the theory 
of complements. 

In the setting of Theorem \ref{t-eff-bir-nefbig-weil}, we need to employ new ideas in addition to the existing 
methods to do induction. We focus on showing that $|mN|$ defines a birational map for some bounded $m$. 
To achieve this we want to find bounded $m$ so that for general closed points  $x,y\in X$ 
(possibly switching them), there is $0\le \Delta\sim_\Q mN-K_X$ such that 
$(X,\Delta)$ is lc at $x$ and $\{x\}$ is an isolated lc centre but $(X,\Delta)$ is not klt at $y$. 
Then using vanishing theorems we get $\alpha\in H^0(X,mN)$ with $\alpha(x)\neq 0$ but $\alpha(y)=0$. 
Thus $|mN|$ defines a birational map. 

To proceed we reduce to the case when $N-K_X$ and $N+K_X$ are big. Replacing $N$ with a bounded multiple 
we can ensure $N-K_X$ is big. However, to ensure $N+K_X$ is big we need to use the BAB [\ref{B-BAB}, Theorem 1.1] 
except in some cases, e.g. when $N$ is a multiple of $-K_X$.

Pick $m$ such that ${\rm vol}(mN-K_X)>(2d)^d$. Initially we take $m$ minimal with this property.
There is $\le \Delta\sim_\Q mN-K_X$ such that 
$(X,\Delta)$ is lc at $x$ with a unique lc centre $G$ but $(X,\Delta)$ is not klt at $y$.
If always $\dim G=0$, then ${|mN|}$ defines a birational map by vanishing theorems as explained above. 
Eventually we need to bound $m$ from above. 

The hard part is when $\dim G>0$. In this case we need to show that 
${\rm vol}(mN|_{G})$ is bounded from below away from zero in order to replace $G$ and decrease dimension. 
There is a kind of adjunction, that is, we can write 
$$
mN|_F\sim_\Q (K_X+\Delta)|_F\sim_\Q K_F+\Theta_F+P_F
$$ 
where $F$ is the normalisation of $G$, $\Theta_F$ is a boundary, and $P_F$ is big. 

Next we reduce to the case when $F$ is birational to a bounded variety $F'$. 
If $(F,\Theta_F+P_F)$ has bad singularities, then we can bound ${\rm vol}(mN|_{G})$ by 
comparing singularities with those on $F'$.
This allows to reduce to the case when $(F,\Theta_F+P_F)$ is $\frac{\epsilon}{2}$-lc. 

To treat the remaining case we show that $N|_F$ 
is integral up to a bounded multiple, hence we can assume $mN|_F$ is integral. 
Since $mN|_F-K_F$ is big, applying induction on dimension, ${|mN|_F|}$ defines a birational map.
Therefore, ${\rm vol}(mN|_{G})\ge 1$ as desired. Carrying out the ideas of 
this paragraph is one of the main innovations of this paper.
Again we eventually need to show $m$ is bounded.

Now we turn to sketch of proofs of \ref{t-lct-nef-big-slc-cy} and \ref{cor-bnd-semi-lc-pol-CY}. 
As mentioned earlier, the former implies the latter, so we will focus on \ref{t-lct-nef-big-slc-cy}. 
Since we want to show that $(X,B+tN)$ is slc for some $t$ bounded from below away from zero, 
we can normalise $X$ and then replace it with any of its irreducible components, hence 
we can assume $(X,B)$ is lc. 

Next replace $(X,B)$ with a $\Q$-factorial dlt model and $N$ with its pullback. It turns out that 
$N$ is still integral. This is a consequence of the fact that $\Supp N$ does not 
contain any non-klt centre of $(X,B)$. 

Pick $\epsilon>0$ sufficiently small. 
Extract all the prime divisors $D$ over $X$ with log discrepancy $a(D,X,0)\le \epsilon$, say via $Y\to X$. 
Replace $X$ with $Y$ and $K_X+B,N$ with their pullbacks. Again it turns out that $N$ is still integral. 
Now $X$ is $\epsilon$-lc, $N$ is integral nef and big, and $N-K_X\sim_\Q N+B$ is big. 

Applying birational boundedness, that is, \ref{t-eff-bir-nefbig-weil}, we find a bounded $m$ such that 
${|mN|}$ defines a birational map. Next we show that $(X,\Supp (B+N))$ is birationally bounded.
Comparing singularities on $X$ and the bounded model ensures that $t$ exists as desired.\\

{{\textbf{\sffamily{Acknowledgements.}}}}
This work was mainly done at Cambridge University with the support of a grant of the Royal Society and partially done while visiting the Yau Mathematical Sciences Center, Tsinghua University in August 2019. Revision was done at Tsinghua University in 2022. Thanks to Yifei Chen, Jingjun Han and participants of a reading workshop at Fudan University, Xiaowei Jiang, and the referee for their helpful comments.


\section{\bf Preliminaries}

We work over an algebraically closed field $k$. All the varieties in this paper are quasi-projective over $k$ unless stated otherwise.

\subsection{Contractions}
By a \emph{contraction} we mean a projective morphism $f\colon X\to Y$ of varieties 
such that $f_*\mathcal{O}_X=\mathcal{O}_Y$ ($f$ is not necessarily birational). In particular, 
$f$ is surjective and has connected fibres.

\subsection{Divisors}
Let $X$ be a variety, and let $M$ be an $\R$-divisor on $X$. 
We denote the coefficient of a prime divisor $D$ in $M$ by $\mu_DM$.  Writing $M=\sum m_iM_i$ where 
$M_i$ are the distinct irreducible components, the notation $M^{\ge a}$ means 
$\sum_{m_i\ge a} m_iM_i$, that is, we ignore the components with coefficient $<a$. One similarly defines $M^{\le a}, M^{>a}$, and $M^{<a}$. 
    
Now let $f\colon X\to Z$ be a morphism to another variety. If $N$ is an $\R$-Cartier 
$\R$-divisor on $Z$, we sometimes denote $f^*N$ by $N|_X$. 

For a birational map $X\bir X'$ (resp. $X\bir X''$)(resp. $X\bir X'''$)(resp. $X\bir Y$) 
whose inverse does not contract divisors, and for 
an $\R$-divisor $M$ on $X$ we usually denote the pushdown of $M$ to $X'$ (resp. $X''$)(resp. $X'''$)(resp. $Y$) 
by $M'$ (resp. $M''$)(resp. $M'''$)(resp. $M_Y$).

Recall that an $\R$-Cartier $\R$-divisor $M$ on a variety $X$ projective over $Z$ 
is said be big if we can write $M\sim_\R A+P$ where $A$ is ample$/Z$ and $P\ge 0$.
When $M$ is $\Q$-Cartier we can replace $\sim_\R$ with $\sim_\Q$ and assume $A,P$ are $\Q$-divisors.

\begin{lem}\label{l-restriction-big-divisors}
Let $f\colon X\to Z$ be a projective morphism between normal varieties. Then there is a non-empty open subset 
$U$ of $Z$ such that if $M$ is any $\R$-Cartier $\R$-divisor on $X$ which is big over $Z$, then 
$M|_F$ is big for every fibre of $f$ over closed points in $U$. 
\end{lem}
\begin{proof}
Let $\phi\colon X'\to X$ be a resolution and let $f'\colon X'\to Z$ be the induced morphism. 
Let $U$ be a non-empty open subset of $Z$ over which $f'$ is smooth. Now let 
$M$ be an $\R$-Cartier $\R$-divisor on $X$ which is big over $Z$. 
Let $F'$ be a fibre 
of $f'$ over a closed point of $U$ and let $G'$ be the generic fibre of $f'$. 

We can write 
$M':=\phi^*M=\sum r_iM_i'$ where $r_i\in \R$ and $M_i$ are Cartier divisors. Moving the $M_i$ 
we can assume that their support do not contain any component of $F'$. Since $M'$ is big over $Z$, we can decrease its  
coefficients slightly so that the resulting divisor, say $L'$, is $\Q$-Cartier and also big over $Z$.
Take $n\in \N$ so that $nL'$ is Cartier. Then for any $m\in \N$ we have 
$$
h^0(mnL'|_{F'})\ge h^0(mnL'|_{G'})
$$  
by the upper-semi-continuity of cohomology, where the second $h^0$ is dimension over the function field of $Z$. 

Since $L'$ is big over $Z$, $h^0(mnL'|_{G'})$ grows like $m^{\dim G'}$, hence 
$h^0(mnL'|_{F'})$ grows like $m^{\dim F'}$ which implies that $L'|_{F'}$ is big. 
Since the support of $M'-L'\ge 0$ does not contain any component of $F'$, we see that $M'|_{F'}$ is big. 
This in turn implies that $M|_F$ is also big where $F$ is the fibre of $f$ corresponding to $F'$.

\end{proof}

\begin{lem}\label{l-perturbation-R-div}
Let $X$ be a normal variety and $M$ be an $\R$-Cartier $\R$-divisor on $X$. Then 
there is a $\Q$-Cartier $\Q$-divisor $A$ such that $\Supp A=\Supp M$ and $M-A$ has 
arbitrarily small coefficients.
\end{lem}
\begin{proof} 
Since $M$ is $\R$-Cartier, we can write $M=\sum_1^l r_iM_i$ where $r_i$ are real numbers and $M_i$ 
are Cartier divisors. Assume that $l$ is minimal. Then the $r_i$ are 
$\Q$-linearly independent: if not, then say $r_1=\sum_2^l \alpha_ir_i$ where $\alpha_i$ are rational numbers, so 
$$
M=\sum_2^l r_i(\alpha_iM_1+M_i),
$$ 
hence we can write $M=\sum_2^lr_i'M_i'$ where $r_i'$ are real numbers and $M_i'$ are Cartier divisors,
contradicting the minimality of $l$. Then $\Supp M_i\subseteq \Supp M$ for 
every $i$: indeed otherwise there is a prime divisor $D$ which is not a component of $M$ 
but it is a component of some $M_i$ which gives $0=\mu_DM=\sum_1^l r_i\mu_DM_i$ producing a 
$\Q$-linear dependence of the $r_i$, a contradiction. 
Now a small perturbation of the $r_i$ gives the desired $\Q$-divisor $A$. 

\end{proof}

The next lemma is not used in the paper but it will be useful elsewhere.

\begin{lem}\label{l-fibres-over-smooth}
Let $f\colon X\to Z$ be a contraction of normal varieties and $M$ an $\R$-divisor on $X$. 
Let $\phi\colon X'\to X$ be a log resolution of $(X,\Supp M)$ and let $M'$ be the 
sum of the birational transform of $M$ and the reduced exceptional divisor of $\phi$. 
Let $F$ be a general fibre of $f$ and $F'$ the corresponding fibre of $X'\to Z$. 
Then $M'|_{F'}$ is the birational transform of $M|_F$ plus the reduced exceptional divisor of $\psi\colon F'\to F$.
\end{lem}
\begin{proof}
Since $F'$ is a general fibre of $X'\to Z$, $\Supp M'$ does not contain $F'$, so $M'|_{F'}$ 
is well-defined as a divisor. However, $M$ may not be $\R$-Cartier but $M|_F$ can be defined as follows. 
Let $U$ be the smooth locus of $X$. Then the complement of $F\cap U$ has codimension $\ge 2$ in $F$, so 
$M|_F$ is well-defined on $F\cap U$ and then we let $M|_F$ be its closure in $F$. 

By assumption $M'=M^\sim+E$ where $M^\sim$ is the birational transform of $M$ and $E$ is 
the reduced exceptional divisor of $\phi$. The exceptional locus of $\psi$ is $E|_{F'}$, so
$M'|_{F'}$ is the sum of $M^\sim|_{F'}$ and the reduced exceptional divisor of $\psi$.  
On the other hand, letting $U'$ be the inverse image of $U$, there is an exceptional divisor $G$ such that 
$\phi^*M|_U= M^\sim|_{U'}+G|_{U'}$. Then $M^\sim|_{F'\cap U'}+G|_{F'\cap U'}$ is the pullback 
of $M|_{F\cap U}$. This implies that the pushdown of $M^\sim|_{F'}$ to $F$ is $M|_F$. Since no component 
of $M^\sim|_{F'}$ is exceptional over $F$, $M^\sim|_{F'}$ is the birational transform of $M|_F$, 
 so the claim follows. 
 
\end{proof}

\subsection{Linear systems}\label{ss-lin-systems}

Let $X$ be a normal variety and let $M$ be an $\R$-divisor on $X$. The round down $\rddown{M}$ determines a 
reflexive sheaf $\mathcal{O}_X(\rddown{M})$.
We usually write $H^i(M)$ instead of $H^i(X,\mathcal{O}_X(\rddown{M}))$ and write $h^i(M)$ for 
$\dim_k H^i(M)$.  We can describe 
$H^0(M)$ in terms of rational functions on $X$ as 
$$
H^0(M)=\{0\neq \alpha\in K \mid\Div(\alpha)+M\ge 0 \}\cup \{0\}
$$
where $K$ is the function field of $X$ and $\Div(\alpha)$ is the divisor associated to $\alpha$.

Assume $h^0(M)\neq 0$. 
The \emph{linear system} $|M|$ is defined as 
$$
|M|=\{N\mid 0\le N\sim M\}=\{\Div(\alpha)+M \mid 0\neq \alpha\in H^0(M)\}.
$$
Note that $|M|$ is not equal to $|\rddown{M}|$ unless $M$ is integral.
The \emph{fixed part} of $|M|$ is the $\R$-divisor $F$ with the property: 
if $G\ge 0$ is an $\R$-divisor and $G\le N$ for every $N\in |M|$, then $G\le F$. In particular, $F\ge 0$. 
We then define the \emph{movable part} of $|M|$ to be $M-F$ which is defined up to 
linear equivalence. If $\langle M\rangle:=M-\rddown{M}$, then the fixed part of $|M|$ is equal to 
$\langle M\rangle$ plus the fixed part of $|\rddown{M}|$. Moreover, if $0\le G\le F$, then 
the fixed and  movable parts of $|M-G|$ are $F-G$ and $M-F$, respectively. 

Note that it is clear from the definition that the movable part of $|M|$ is an integral divisor 
but the fixed part is only an $\R$-divisor.

\subsection{Pairs and singularities}\label{ss-pairs}
A \emph{sub-pair} $(X,B)$ consists of a normal variety $X$ and an $\R$-divisor 
$B$ such that $K_X+B$ is $\R$-Cartier. If $B\ge 0$, we call $(X,B)$ a \emph{pair} and 
if the coefficients of $B$ are in $[0,1]$ we call $B$ a \emph{boundary}. 

Let $\phi\colon W\to X$ be a log resolution of  a sub-pair $(X,B)$. Let $K_W+B_W$ be the 
pullback of $K_X+B$. The \emph{log discrepancy} of a prime divisor $D$ on $W$ with respect to $(X,B)$ 
is defined as 
$$
a(D,X,B):=1-\mu_DB_W.
$$
We say $(X,B)$ is \emph{sub-lc} (resp. \emph{sub-klt})(resp. \emph{sub-$\epsilon$-lc}) if 
 $a(D,X,B)$ is $\ge 0$ (resp. $>0$)(resp. $\ge \epsilon$) for every $D$. This means that  
every coefficient of $B_W$ is $\le 1$ (resp. $<1$)(resp. $\le 1-\epsilon$). If $(X,B)$ is a pair, 
we remove the sub and just say it is lc (resp. klt)(resp. $\epsilon$-lc). 
Note that since $a(D,X,B)=1$ for most prime divisors, we necessarily have $\epsilon\le 1$.

Let $(X,B)$ be a sub-pair. A \emph{non-klt place} of $(X,B)$ is a prime divisor $D$ over $X$, that is, 
on birational models of $X$,  such that $a(D,X,B)\le 0$, and a \emph{non-klt centre} is the image of such a $D$ on $X$. 
An \emph{lc place} of $(X,B)$ is a prime divisor $D$ over $X$ such that $a(D,X,B)=0$, and 
an \emph{lc centre} is the image on $X$ of an lc place. When $(X,B)$ is lc, then non-klt places and centres 
are the same as lc centres and places.

A \emph{log smooth} sub-pair is a sub-pair $(X,B)$ where $X$ is smooth and $\Supp B$ has simple 
normal crossing singularities. Assume $(X,B)$ is a log smooth pair and assume $B=\sum_1^r B_i$ is reduced  
where $B_i$ are the irreducible components of $B$. 
A \emph{stratum} of $(X,B)$ is a component of $\bigcap_{i\in I}B_i$ for some $I\subseteq \{1,\dots,r\}$.  
Since $B$ is reduced, a stratum is nothing but an lc centre of $(X,B)$.

\subsection{Semi-log canonical pairs}\label{ss-slc-pairs}
A \emph{semi-log canonical (slc) pair} $(X,B)$ over a field $K$ of characteristic zero 
(not necessarily algebraically closed) consists of a 
reduced pure dimensional quasi-projective scheme $X$ over $K$ and an 
$\R$-divisor $B\ge 0$ on $X$ satisfying the following conditions: 
\begin{itemize}
\item $X$ is $S_2$ with nodal codimension one singularities,

\item no component of $\Supp B$ is contained in the singular locus of $X$,

\item $K_X+B$ is $\R$-Cartier, and 

\item if $\pi\colon X^\nu\to X$ is the normalisation of $X$ and $B^\nu$ is the sum of the birational 
transform of $B$ and the conductor divisor of $\pi$, then $(X^\nu,B^\nu)$ is lc.
\end{itemize}
By $(X^\nu,B^\nu)$ being lc we mean after passing to the algebraic closure of $K$, 
$(X^\nu,B^\nu)$ is an lc pair on each of its irreducible components 
(we can also define being lc over $K$ directly using discrepancies as in \ref{ss-pairs}).
The conductor divisor of $\pi$ is the sum of the prime divisors on $X^\nu$ whose 
images on $X$ are contained in the singular locus of $X$. It turns out that $K_{X^\nu}+B^\nu=\pi^*(K_X+B)$ 
for a suitable choice of $K_{X^\nu}$ in its linear equivalence class: 
to see this note that $X$ is Gorenstein outside a codimension $\ge 2$ closed subset, so shrinking $X$ 
we can assume it is Gorenstein and that $X^\nu$ is regular; in this case $B$ is $\R$-Cartier so 
we can remove it in which case the equality follows from 
[\ref{Kollar-singMMP}, 5.7]. 
See  [\ref{Kollar-singMMP}, Chapter 5] for more on slc pairs.

\subsection{b-divisors}
A \emph{b-$\R$-Cartier b-divisor} over a variety $X$ is the choice of  
a projective birational morphism 
$Y\to X$ from a normal variety and an $\R$-Cartier $\R$-divisor $M$ on $Y$ up to the following equivalence: 
 another projective birational morphism $Y'\to X$ from a normal variety and an $\R$-Cartier $\R$-divisor
$M'$ define the same b-$\R$-Cartier  b-divisor if there is a common resolution $W\to Y$ and $W\to Y'$ 
on which the pullbacks of $M$ and $M'$ coincide.  

A b-$\R$-Cartier  b-divisor  represented by some $Y\to X$ and $M$ is \emph{b-Cartier} if  $M$ is 
b-Cartier, i.e. its pullback to some resolution is Cartier.

\subsection{Generalised pairs}\label{ss-gpp}
A \emph{generalised pair} consists of 
\begin{itemize}
\item a normal variety $X$ equipped with a projective
morphism $X\to Z$, 

\item an $\R$-divisor $B\ge 0$ on $X$, and 

\item a b-$\R$-Cartier  b-divisor over $X$ represented 
by some projective birational morphism $X' \overset{\phi}\to X$ and $\R$-Cartier $\R$-divisor
$M'$ on $X$
\end{itemize}
such that $M'$ is nef$/Z$ and $K_{X}+B+M$ is $\R$-Cartier,
where $M:= \phi_*M'$. 

We refer to $M'$ as the nef part of the pair.  
Since a b-$\R$-Cartier b-divisor is defined birationally, 
in practice we will often replace $X'$ with a resolution and replace $M'$ with its pullback.
When $Z$ is a point we drop it but say the pair is projective. 

Now we define generalised singularities.
Replacing $X'$ we can assume $\phi$ is a log resolution of $(X,B)$. We can write 
$$
K_{X'}+B'+M'=\phi^*(K_{X}+B+M)
$$
for some uniquely determined $B'$. For a prime divisor $D$ on $X'$ the \emph{generalised log discrepancy} 
$a(D,X,B+M)$ is defined to be $1-\mu_DB'$. 

We say $(X,B+M)$ is 
\emph{generalised lc} (resp. \emph{generalised klt})(resp. \emph{generalised $\epsilon$-lc}) 
if for each $D$ the generalised log discrepancy $a(D,X,B+M)$ is $\ge 0$ (resp. $>0$)(resp. $\ge \epsilon$).

For the basic theory of generalised pairs see [\ref{BZh},Section 4].

\subsection{Minimal models, Mori fibre spaces, and MMP}
Let $ X\to Z$ be a
projective morphism of normal varieties and $D$ be an $\R$-Cartier $\R$-divisor
on $X$. Let $Y$ be a normal variety projective over $Z$ and $\phi\colon X\bir Y/Z$
be a birational map whose inverse does not contract any divisor. 
Assume $D_Y:=\phi_*D$ is also $\R$-Cartier and that 
there is a common resolution $g\colon W\to X$ and $h\colon W\to Y$ such that
$E:=g^*D-h^*D_Y$ is effective and exceptional$/Y$, and
$\Supp g_*E$ contains all the exceptional divisors of $\phi$.

Under the above assumptions we call $Y$ 
a \emph{minimal model} of $D$ over $Z$ if $D_Y$ is nef$/Z$.
On the other hand, we call $Y$ a \emph{Mori fibre space} of $D$ over $Z$ if there is an extremal contraction
$Y\to T/Z$ with $-D_Y$  ample$/T$ and $\dim Y>\dim T$.

If one can run a \emph{minimal model program} (MMP) on $D$ over $Z$ which terminates 
with a model $Y$, then $Y$ is either a minimal model or a Mori fibre space of 
$D$ over $Z$. If $X$ is a Mori dream space, 
eg if $X$ is of Fano type over $Z$, then such an MMP always exists by [\ref{BCHM}].

\subsection{Potentially birational divisors}
Let $X$ be a normal projective variety and let $D$ be a big
$\Q$-Cartier $\Q$-divisor on $X$. We say that $D$ is \emph{potentially
birational} if for any pair $x$ and $y$ of general closed points of $X$, possibly
switching $x$ and $y$, we can find $0 \le  \Delta \sim_\Q (1 - \epsilon)D$ 
for some $0 < \epsilon < 1$ such that 
$(X, \Delta)$ is not klt at $y$ but $(X, \Delta)$ is lc at $x$ and
$\{x\}$ is an isolated non-klt centre. Note that this definition is stronger than that of [\ref{HMX2}, Definition 3.5.3].

A useful property of potentially birational divisors is that 
if $D$ is potentially birational, then $|K_X+\lceil D\rceil|$ defines a birational map 
[\ref{HMX1}, Lemma 2.3.4].

\subsection{Bounded families of pairs}\label{ss-bnd-couples}
We say a set $\mathcal{Q}$ of normal projective varieties is \emph{birationally bounded} (resp. \emph{bounded}) 
if there exist finitely many projective morphisms $V^i\to T^i$ of varieties  
such that for each $X\in \mathcal{Q}$ there exist an $i$, a closed point $t\in T^i$, and a 
birational isomorphism (resp. isomorphism) $\phi\colon V^i_t\bir X$  where $V_t^i$ is the fibre of $V^i\to T^i$ over $t$.

Next we will define boundedness for couples. 
A \emph{couple} $(X,S)$ consists of a normal projective variety $X$ and a  divisor 
$S$ on $X$ whose coefficients are all equal to $1$, i.e. $S$ is a reduced divisor. 
We use the term couple instead of pair because  
$K_X+S$ is not assumed $\Q$-Cartier and $(X,S)$ is not assumed to have good singularities.
 
We say that a set $\mathcal{P}$ of couples  is \emph{birationally bounded} if there exist 
finitely many projective morphisms $V^i\to T^i$ of varieties and reduced divisors $C^i$ on $V^i$ 
such that for each $(X,S)\in \mathcal{P}$ there exist an $i$, a closed point $t\in T^i$, and a 
birational isomorphism $\phi\colon V^i_t\bir X$ such that $(V^i_t,C^i_t)$ is a couple and 
$E\le C_t^i$ where $V_t^i$ and $C_t^i$ are the fibres over $t$ of the morphisms $V^i\to T^i$ 
and $C^i\to T^i$, respectively, and $E$ is the sum of the 
birational transform of $S$ and the reduced exceptional divisor of $\phi$.
We say $\mathcal{P}$ is \emph{bounded} if we can choose $\phi$ to be an isomorphism. 
  
A set $\mathcal{R}$ of projective pairs $(X,B)$ is said to be {log birationally bounded} (resp. log {bounded}) 
if the set of the corresponding couples $(X,\Supp B)$ is birationally bounded (resp. bounded).
Note that this does not put any condition on the coefficients of $B$, e.g. we are not requiring the 
coefficients of $B$ to be in a finite set.

\subsection{Families of subvarieties}\label{ss-cov-fam-subvar}
Let $X$ be a normal projective variety. A \emph{bounded family $\mathcal{G}$ of subvarieties} of $X$ 
is a family of (closed) subvarieties such that there are finitely many morphisms 
$V^i\to T^i$ of projective varieties together with morphisms $V^i\to X$ 
such that $V^i\to X$ embeds in $X$ the fibres of $V^i\to T^i$ over closed points, and 
each member of the family $\mathcal{G}$ is isomorphic to a fibre of 
some $V^i\to T^i$ over some closed point. Note that we can replace the $V^i\to T^i$ 
so that we can assume the set of points of $T^i$ corresponding to members of $\mathcal{G}$ 
is dense in $T^i$.
We say the family $\mathcal{G}$ is a \emph{covering family of subvarieties} of $X$ if 
the union of its members contains some non-empty open subset of $X$. In particular, this means 
$V^i\to X$ is surjective for at least one $i$.
When we say $G$ is a \emph{general member of $\mathcal{G}$} 
we mean there is $i$ such that $V^i\to X$ is surjective,  the set $A$ of points of $T^i$ corresponding 
to members of $\mathcal{G}$ is dense in $T^i$, and $G$ is the fibre of 
$V^i\to T^i$ over a general point of $A$ (in particular, $G$ is among the general fibres of 
$V^i\to T^i$). Note that the definition of a bounded family here is compatible with 
 \ref{ss-bnd-couples}.

\subsection{Creating non-klt centres}\label{ss-non-klt-centres}
In this subsection we make some preparations on non-klt centres.

(1) 
First we need the following lemma.

\begin{lem}\label{l-unique-lc-place}
Assume that 
\begin{itemize}
\item $(X,B)$ is a projective pair where $B$ is a $\Q$-divisor,

\item $\Delta\ge 0$ is a $\Q$-Cartier $\Q$-divisor and $H$ is an ample $\Q$-divisor,
 
\item $x,y\in X$ are closed points, 

\item $(X,B)$ is klt near $x$ and $(X,B+\Delta)$ is lc near $x$ with a non-klt
centre $G$ containing $x$ but $(X,B+\Delta)$ is not klt near $y$,

\item $G$ is minimal among the non-klt centres of $(X,B+\Delta)$ containing $x$, and 

\item either $y\in G$ or $(X,B+\Delta)$ has a non-klt centre containing $y$ but not $x$.
\end{itemize}
Then there exist rational numbers $0\le t\ll s\le 1$ and  
a $\Q$-divisor $0\le E\sim_\Q tH$ such that $(X,B+s\Delta+E)$ is not klt near $y$ but it is lc near $x$ with 
a unique non-klt place whose centre contains $x$, and the centre of this non-klt place is $G$.   
\end{lem}
\begin{proof}
Pick $0\le M\sim_\Q H$ whose support contains $G$ so that for any $s<1$ any non-klt centre of $(X,B+s\Delta+M)$ 
 passing through $x$ is contained in $G$: this is possible by taking a general member 
$mM$ of the sublinear system of $|mH|$ consisting of the elements that contain $G$, for some sufficiently large 
natural number $m$.  On the other hand, pick $0\le N\sim_\Q H$
such that $x\notin \Supp N$. If $y\notin G$, then pick $N$ so
that $N$ contains a non-klt centre $P$ of $(X,B+\Delta)$ with $y\in P$ but $x\notin P$: 
this is possible by the last condition of the lemma.

Let $\phi\colon W\to X$ be a log resolution of 
$$
(X,B+\Delta+M+N).
$$ 
Then $\phi^*H\sim_\Q A+C$ where $A\ge 0$ is ample and $C\ge 0$. Let 
$C'=\phi_*C$ and $H'=\phi_*(A+C)$. We can assume that $G\subseteq \Supp C'$.  Replacing $X$ with a higher resolution we can assume $\phi$ is a log 
resolution of 
$$
(X,B+\Delta+M+N+C');
$$ 
note that here we pull back $A,C$ to the new resolution, so $A$ may no longer be 
ample but it is nef and big, hence perturbing coefficients in the exceptional components 
we can make $A$ ample again. Changing $A$ 
up to $\Q$-linear equivalence we can assume $A$ is general, so $\phi$ is a log resolution of 
$$
(X,B+\Delta+M+N+H').
$$

Write 
$$
K_W+\Gamma_{a,b,c,d}=\phi^*(K_{X}+B+a\Delta+bM+cN+dH').
$$ 
Let $T$ be the sum of the components of  $\rddown{\Gamma_{1,0,0,0}^{\ge 0}}$
whose image on $X$ is $G$ where $\Gamma_{a,b,c,d}^{\ge 0}$  denotes the effective part of $\Gamma_{a,b,c,d}$. 
We can assume $T\neq 0$ since $G$ is a non-klt centre of $(X,B+\Delta)$.

Pick a rational number $0<b\ll 1$ and let 
$a$ be the lc threshold of $\Delta$ with respect to $(X,B+bM)$ near $x$. Then $a$ is sufficiently 
close to $1$ but not equal to $1$. Moreover, 
$$
\rddown{\Gamma_{a,b,0,0}^{\ge 0}}\subseteq \rddown{\Gamma_{1,b,0,0}^{\ge 0}}=\rddown{\Gamma_{1,0,0,0}^{\ge 0}}.
$$  
By our choice of $M$ and by the minimality of $G$, the only possible non-klt centre of 
$(X,B+a\Delta+bM)$ through $x$ is $G$, so 
any component of $\rddown{\Gamma_{a,b,0,0}^{\ge 0}}$ whose image contains $x$, is a component of $T$. 

Now pick a rational number $0<d\ll b$ and let $\lambda$ be the lc threshold of $a\Delta+bM$ 
with respect to $(X,B+dH')$ near $x$. Then $\lambda$ is sufficiently 
close to $1$ but not equal to $1$  as $G\subseteq \Supp H'$. Moreover, 
$$
\rddown{\Gamma_{\lambda a,\lambda b,0,d}^{\ge 0}}\subseteq \rddown{\Gamma_{a,b,0,d}^{\ge 0}}
\subseteq \rddown{\Gamma_{a,b,0,0}^{\ge 0}},
$$    
so any component of $\rddown{\Gamma_{\lambda a,\lambda b,0,d}^{\ge 0}}$  whose image contains $x$ is a component of $T$. 
Thus $G$ is the only non-klt centre of 
$$
(X,B+\lambda a\Delta+\lambda bM+dH')
$$ 
passing through $x$.
In particular, there is a component $S$ of $\rddown{\Gamma_{\lambda a,\lambda b,0,d}^{\ge 0}}$  
whose image is $G$.
Now noting that $\phi^*H'=A+C$ and that 
$$
\Gamma_{\lambda a,\lambda b,0,d}=\Gamma_{\lambda a,\lambda b,0,0}+dA+dC,
$$ 
possibly after perturbing the coefficients of $C$ and replacing $A$ accordingly, 
we can assume that $S$ is the only component of $\rddown{\Gamma_{\lambda a,\lambda b,0,d}^{\ge 0}}$ 
whose image contains $x$.

Let $s=\lambda a$, $t=\lambda b+d$, and $E=\lambda bM+dH'$. 
By construction, $(X,B+s\Delta+E)$ is lc near $x$ with 
a unique non-klt place whose centre contains $x$, and the centre of this non-klt place is $G$. 
If $(X,B+s\Delta +E)$ is not klt near $y$, then  
 we are done. In particular, this is the case if $y\in G$. Thus we can assume 
$(X,B+s\Delta +E)$ is klt near $y$ and that $y\notin G$.  
So by assumption, there is a non-klt centre $P$ of 
$(X,B+\Delta)$ containing $y$ but not $x$. By our choice of $N$, $P\subset \Supp N$ but $x\notin \Supp N$. 
Let $c$ be the lc threshold of $N$ with respect to $(X,B+s\Delta +E)$ near $y$. Since 
$s$ is sufficiently close to $1$, $c$ is sufficiently small. Now replace $E$ with $E+cN$ 
and replace $t$ with $t+c$.
 
\end{proof}

(2)
Let $X$ be a normal projective variety of dimension $d$ and $D$ a nef and big $\Q$-divisor. 
Assume $\vol(D)>(2d)^d$. Then there is a bounded family of subvarieties of $X$ such that 
for each pair $x,y\in X$ of general closed points, there is a member $G$ of the family and 
there is $0\le \Delta\sim_\Q D$ such that $(X,\Delta)$ is lc near $x$ with a unique non-klt place 
whose centre contains $x$, that centre is $G$, 
and $(X,\Delta)$ is not klt at $y$ [\ref{HMX2}, Lemma 7.1]. 

If in addition we are also given a $\Q$-divisor $0\le M\sim_\Q D$, then we can assume $\Supp M\subset \Supp \Delta$ 
simply by adding a small multiple of $M$ to $\Delta$. Since $x,y$ are general, they are not 
contained in $\Supp M$.\\

(3)
\begin{lem}\label{l-non-klt-centres-i}
Under the setting of (2) assume that $A$ is an ample $\Q$-divisor. 
Let $\Delta$ and $G$ be chosen for a pair $x,y\in X$ of general closed points and 
assume $\dim G>0$ and $\vol(A|_G)> d^d$. Then either for $z_1=x, z_2=y$ or for 
$z_1=y, z_2=x$ we can find a rational number $0<c<2$ and  
a $\Q$-divisor $0\le L\sim_\Q cA$ such that 
\begin{itemize}
\item $(X,\Delta+L)$ is not klt near $z_2$ but it is lc near $z_1$,
\item  $(X,\Delta+L)$ has a minimal non-klt centre $G'\subsetneq G$ through $z_1$, and 
\item either $z_2\in G'$ or $(X,\Delta+L)$ has a non-klt centre containing $z_2$ but not $z_1$.
\end{itemize}
\end{lem}
\begin{proof}
 By [\ref{kollar-sing-pairs}, 6.8.1 and its proof], 
there exist a rational number $0<e<1$ and a $\Q$-divisor $0\le N\sim_\Q eA$ such that 
$(X,\Delta+N)$ is lc near $x$ but has a non-klt centre through $x$ other than $G$. 
Since intersection of non-klt centres near $x$ is a union of such centres, 
the minimal non-klt centre of $(X,\Delta+N)$ at $x$ is a proper subvariety ${G}'\subsetneq G$.

If $y\in G'$ or if $(X,\Delta+N)$ has a non-klt centre containing $y$ but not $x$, 
then we let $z_1=x,z_2=y$ and $c=e, L=N$. Thus we 
can assume that $y\notin G'$ and that every non-klt centre of $(X,\Delta+N)$ containing $y$ also contains $x$. In particular, this implies that $(X,\Delta+N)$ is lc at $y$. Moreover,   
 $(X,\Delta)$ is lc at $y$ and $G$ is the only non-klt centre of $(X,\Delta)$ containing $y$ as it is the only non-klt centre of $(X,\Delta)$ containing $x$. 

Let $G''$ be the minimal non-klt centre of $(X,\Delta+N)$ at $y$. Then $G''\subseteq G$ and $G''$ contains $x$. If $G''$ is a proper subvariety of $G$, then we let $z_1=y, z_2=x$ and switch $G',G''$ and again let $c=e,L=N$.  

Thus we can assume $G''=G$. Then since $G$ is a non-klt centre of $(X,\Delta)$ at $y$,  we see that 
$G\not\subseteq \Supp N$, hence $G$ is the unique non-klt centre $(X,\Delta+N)$ at $y$.
Then $(X,\Delta+(1+\epsilon)N)$ is lc at $y$ but not lc at $x$ for a small $\epsilon>0$ as $G'\subset \Supp N$. 
Here we apply [\ref{kollar-sing-pairs}, 6.8.1 and its proof] to find a rational number 
$0<e'<1$ and a $\Q$-divisor $0\le N'\sim_\Q e'A$ such that $(X,\Delta+N+N')$ is 
lc at $y$ with a minimal non-klt centre $G'''\subsetneq G$ at $y$ but not lc at $x$. 
This time we let $z_1=y, z_2=x$ and switch $G',G'''$ and let $c=e+e',L=N+N'$.

\end{proof}

\begin{lem}\label{l-non-klt-centres-ii}
Under the setting of (2) assume that $A$ is a nef and big $\Q$-divisor. 
Let $\Delta$ and $G$ be chosen for a pair $x,y\in X$ of general closed points and 
assume $\dim G>0$ and $\vol(A|_G)> d^d$. Then possibly switching $x,y$, there exist a $\Q$-divisor
$$
0\le {\Delta}^{(1)}\sim_\Q D+2A
$$ 
and a proper subvariety ${G}^{(1)}\subsetneq G$ such that  
$(X,{\Delta}^{(1)})$ is lc near $x$ with a unique non-klt place 
whose centre contains $x$, that centre is ${G}^{(1)}$, and $(X,{\Delta}^{(1)})$ is not klt at $y$ 
(compare with [\ref{HMX1}, Theorem 2.3.5]). 
\end{lem}
\begin{proof}
First note that since we are concerned with general points of $X$, to 
prove the lemma we can replace $X$ with a resolution and replace $D,\Delta,A$ with their pullbacks 
and replace $G$ with its birational transform, hence assume $X$ is smooth. 

First assume that $A$ is ample. By Lemma \ref{l-non-klt-centres-i}, either for $z_1=x, z_2=y$ or for 
$z_1=y, z_2=x$ we can find a rational number $0<c<2$ and  
a $\Q$-divisor $0\le L\sim_\Q cA$ such that 
\begin{itemize}
\item $(X,\Delta+L)$ is not klt near $z_2$ but it is lc near $z_1$,
\item  $(X,\Delta+L)$ has a minimal non-klt centre $G'\subsetneq G$ through $z_1$, and 
\item either $z_2\in G'$ or $(X,\Delta+L)$ has a non-klt centre containing $z_2$ but not $z_1$.
\end{itemize}
If $z_1=x, z_2=y$, then we apply Lemma \ref{l-unique-lc-place} to $(X,\Delta+L)$, $A$, $G'$, to find  rational numbers $0<t\ll s<1$ and  
a $\Q$-divisor $0\le L'\sim_\Q tA$ such that $(X,s\Delta+sL+L')$ is not klt near $y$ 
but it is lc near $x$ with a unique non-klt place whose centre contains $x$, 
and the centre of this non-klt place is $G'$. We then let 
$$
\Delta^{(1)}=s\Delta+sL+L'+ (1-s)M+(2-t-sc)A\sim_\Q \Delta+2A
$$
and let $G^{(1)}=G'$. On the other hand, if $z_1=y, z_2=x$, then we switch $x,y$ and again apply Lemma \ref{l-unique-lc-place} to $(X,\Delta+L)$, $A$, $G'$ and proceed as before. 

It remains to treat the case when $A$ is nef and big but not necessarily ample. 
Write $A\sim_\Q H+P$ independent of $x,y$ 
where $H\ge 0$ is ample and $P\ge 0$ and these are $\Q$-divisors. 
Since $x,y$ are general, they are not contained in $P$. 
Moreover, for any rational number $t\in (0,1)$, we have 
$$
A\sim_\Q (1-t)A+tH+tP
$$ 
where $(1-t)A+tH$ is ample.  
Now  taking $t$ small enough we have 
$$
\vol(((1-t)A+tH)|_G)>d^d,
$$ 
so we can apply the above arguments to $(1-t)A+tH$ to construct $\Delta^{(1)}$ 
and $G^{(1)}$ and then add $2tP$ to $\Delta^{(1)}$.\\

\end{proof}


\section{\bf Geometry of non-klt centres}

In this section we establish some results around the geometry of non-klt centres which are crucial for later sections.

\subsection{Definition of adjunction}\label{constr-adjunction-non-klt-centre}
First we recall the definition of adjunction which was introduced in [\ref{HMX1}]. 
We follow the presentation in [\ref{B-compl}].
Assume the following setting:
\begin{itemize}
\item $(X,B)$ is a projective klt pair, 

\item $G\subset X$ is a subvariety with normalisation $F$,

\item $X$ is $\Q$-factorial near the generic point of $G$,

\item $\Delta\ge 0$ is an $\R$-Cartier divisor on $X$, and 

\item  $(X, B + \Delta)$ is lc near the generic point of $G$, 
and there is a unique non-klt place of this pair with centre $G$ (but the pair may have other non-klt places whose centres are not $G$).
\end{itemize}\

We will define an $\R$-divisor $\Theta_F$ on $F$ with coefficients in $[0,1]$ giving an 
 adjunction formula
$$
K_F + \Theta_F+P_F\sim_\R (K_X + B + \Delta)|_F
$$
where in general $P_F$ is determined only up to $\R$-linear equivalence. Moreover,  
we will see that if the coefficients of $B$ are contained in a fixed DCC set $\Phi$, then the coefficients of $\Theta_F$ 
are also contained in a fixed DCC set $\Psi$ depending only on $\dim X$ and $\Phi$ [\ref{HMX2}, Theorem 4.2]. 

Let $\Gamma$ be the sum of $(B+\Delta)^{<1}$ and the support of $(B+\Delta)^{\ge 1}$. 
Put 
$$
N=B+\Delta-\Gamma
$$ 
which is supported in $\rddown{\Gamma}$.
Let $\phi\colon W\to X$ be a log resolution of $(X,B+\Delta)$ and let $\Gamma_W$ be the sum of the reduced exceptional 
divisor of $\phi$ and the birational transform of $\Gamma$. Let 
$$
N_W=\phi^*(K_X+B+\Delta)-(K_W+\Gamma_W).
$$
Then $\phi_*N_W=N\ge 0$ and $N_W$ is supported in $\rddown{\Gamma_W}$. 
Now run an MMP$/X$ on $K_W+\Gamma_W$ with scaling of some 
ample divisor. We reach a model $Y$ on which $K_Y+\Gamma_Y$ is a limit of movable$/X$ $\R$-divisors. 
Applying the general negativity lemma (cf. [\ref{B-lc-flips}, Lemma 3.3]), we deduce $N_Y\ge 0$. 
In particular, if $U\subseteq X$ is the largest open subset where $(X,B+\Delta)$ is lc, then 
$N_Y=0$ over $U$ and $(Y,\Gamma_Y)$ is a $\Q$-factorial dlt model of $(X,B+\Delta)$ over $U$. 
By assumption, $(X,B+\Delta)$ is lc but not klt at the generic point of $G$. By 
[\ref{B-compl}, Lemma 2.33], no non-klt centre of the pair contains $G$ apart from $G$ 
itself, hence we can assume there is a unique
component $S$ of $\rddown{\Gamma_Y}$ whose image on $X$ contains $G$, and that this image is $G$. 
Moreover,  $G$ is not inside the image of $N_Y$.

Let $h\colon S\to F$ be the morphism induced by $S\to G$. By [\ref{B-compl}, Lemma 2.33], $h$ is a contraction.
By divisorial adjunction we can write 
$$
K_S+\Gamma_S+N_S=(K_Y+\Gamma_Y+N_Y)|_S\sim_\R 0/F
$$
where $N_S=N_Y|_S$ is vertical over $F$. 

If $S$ is exceptional over $X$, then
let $\Sigma_Y$ be the sum of the exceptional$/X$ divisors on $Y$ plus 
the birational transform of $B$. Otherwise let $\Sigma_Y$ 
be the sum of the exceptional$/X$ divisors on $Y$ plus 
the birational transform of $B$ plus $(1-\mu_GB)S$. In any case, $S$ is a 
component of $\rddown{\Sigma_Y}$ and $\Sigma_Y\le \Gamma_Y$.
Applying adjunction again we get 
$K_S+\Sigma_S=(K_Y+\Sigma_Y)|_S$. Obviously $\Sigma_S\le \Gamma_S$. 

Now we define $\Theta_F$: 
for each prime divisor $D$ on $F$, let $t$ be the lc threshold of $h^*D$ with respect to 
$(S,\Sigma_S)$ over the generic point of $D$, and then let $\mu_D\Theta_F:=1-t$. 
Note that $h^*D$ is defined only over the generic 
point of $D$ as $D$ may not be $\Q$-Cartier. 

Note that if the coefficients of $B$ are contained in a fixed DCC set $\Phi$ including $1$, then the 
coefficients of $\Gamma_Y$ are contained in $\Phi$ which in turn implies that 
the coefficients of $\Gamma_S$ are contained in a fixed DCC set $\Psi'$. Applying the ACC 
for lc thresholds [\ref{HMX2}, Theorem 1.1] shows that the coefficients of $\Theta_F$ 
are also contained in a fixed DCC set $\Psi$ depending only on $\dim X$ and $\Phi$. 
Moreover, it turns out that $P_F$ is pseudo-effective 
(see [\ref{HMX2}, Theorem 4.2] and [\ref{B-compl}, Theorem 3.10]).

\subsection{Boundedness of singularities on non-klt centres}
The next result is a generalisation of [\ref{B-compl}, Proposition 4.6] which puts strong restrictions on 
singularities that can appear on non-klt centres under suitable assumptions.
We will not need the proposition in its full generality in this paper but it will likely be useful elsewhere.

\begin{prop}\label{p-bnd-sing-on-non-klt-centre}
Let $d,v$ be natural numbers and $\epsilon'<\epsilon$ be positive real numbers. 
Then there exists a positive real number $t$ 
depending only on $d,v,\epsilon,\epsilon'$ satisfying the following. Assume 
$X,C,M,S,\Delta$, $G,F,\Theta_F,P_F$ 
are as follows:
\begin{enumerate}
\item $(X,C)$ is a projective $\epsilon$-lc pair of dimension $d$,

\item $K_X$ is $\Q$-Cartier, 

\item $M\ge 0$ is a nef $\Q$-divisor on $X$ such that $|M|$ defines a birational map,

\item $S\ge 0$ is an $\R$-Cartier $\R$-divisor on $X$, 

\item the coefficients of $C+S$ are in $\{0\}\cup [\epsilon,\infty)$ and the coefficients of 
$M+C+S$ are in $[1,\infty)$, 

\item $G$ is a general member of a covering family of subvarieties of $X$, with normalisation $F$, 

\item $\Delta\ge 0$ is an $\R$-Cartier $\R$-divisor on $X$, 

\item $(X,\Delta)$ is lc near the generic point of $G$ with a unique non-klt place with centre $G$ (but the pair may have other non-klt places whose centres are not $G$), 

\item the adjunction formula
$$
K_F+\Theta_F+P_F\sim_\R (K_X+\Delta)|_F
$$
is as in \ref{constr-adjunction-non-klt-centre} assuming $P_F\ge 0$,

\item $\vol(M|_F)\le v$, 

\item  $K_X+C+\Delta$ is nef, 

\item $M-(K_X+C+S+\Delta)$ is big, and 

\item $tM-\Delta-S$ is big.\\
\end{enumerate}
Then  for any $0\le L_F\sim_\R M|_F$, the pair 
$$
(F,\Theta_F+P_F+C|_F+tL_F)
$$ 
is $\epsilon'$-lc.
\end{prop}
\begin{proof}
As pointed out in \ref{constr-adjunction-non-klt-centre}, the
coefficients of $\Theta_{F}$ are in a fixed DCC set $\Psi$ depending only on $d$.  
Let $\mathcal{P}$ be the set of couples and $c$ the number given by [\ref{B-compl}, Proposition 4.4], 
for the data 
$$
\dim F,v,\lambda:=\min\Psi^{>0}\cup \{\epsilon\}.
$$ 
Let $\delta>0$ be the number given by [\ref{B-compl}, Proposition 4.2] for $\mathcal{P},\epsilon$. 
Let $l\in\N$ be the smallest number such that $\frac{l-1}{l}>\frac{\epsilon'}{\epsilon}$; 
such $l$ exists because $\frac{\epsilon'}{\epsilon}<1$.
Let $t=\frac{\delta}{2lc}$. We will show that this $t$ satisfies the proposition.  
Note that $t$ depends on $\delta,l,c$ which in turn depend on $\mathcal{P},c,\epsilon,\epsilon'$ 
and these in turn depend on $d,v,\epsilon,\epsilon'$.\\

\emph{Step 1.}
In this step we introduce some basic notation.
We will assume $\dim G>0$ otherwise the statement is vacuous.  
Note that we are assuming that $C,M,S$ are independent of $G$ so 
$G$ being general it is not contained in $\Supp (M+C+S)$ (however, $\Delta$ depends on $G$ 
whose support contains $G$). On the other hand, by [\ref{B-compl}, Lemma 2.6], 
there is a log resolution $\phi\colon W\to X$ of 
$$
(X,\Supp (M+C+S))
$$ 
such that we can write 
$$
M_W:=\phi^*M= A_{W}+R_{W}
$$ 
where $A_W$ is the movable part of $|M_W|$, $|A_W|$ is based point free defining a birational contraction, 
and $R_W\ge 0$ is the fixed part.\\

\emph{Step 2.}
In this step we have a closer look at the adjunction formula given in the statement,  and the related divisors. 
First note that since $G$ is a general member of a covering family, 
$X$ is smooth near the generic point of $G$.

By [\ref{HMX2}, Theorem 4.2][\ref{B-compl}, Lemma 3.12] (by taking $B=0$), we can write $K_F+\Lambda_F=K_X|_F$ 
where $(F,\Lambda_{F})$ is sub-klt and 
$\Lambda_{F}\le \Theta_{F}$. On the other hand, since $G$ is not contained in $\Supp C$, 
the unique non-klt place of $(X,\Delta)$ with centre $G$ is also a unique non-klt place of 
$(X,C+\Delta)$ whose centre is $G$. Thus 
applying [\ref{B-compl}, Lemma 3.12]
once more (this time by taking $B=C$), we can write $K_F+\tilde{C}_F=(K_X+C)|_F$ where 
$(F,\tilde{C}_{F})$ is sub-$\epsilon$-lc. Note that  
$$
\tilde{C}_F=\Lambda_F+C|_F\le \Theta_F+C|_F.
$$\

\emph{Step 3.} 
Let $M_F:=M|_F$, $C_F:=C|_F$, and $S_F:=S|_F$.
In this step we show 
$$
(F, \Supp (\Theta_F+C_F+S_F+M_F))
$$ 
is log birationally bounded using [\ref{B-compl}, Proposition 4.4].
Since $G$ is a general member of a covering family, we can choose a log resolution $F'\to F$ of the above pair  
such that we have an induced morphism $F'\to W$ and that $|A_{F'}|$ defines a birational contraction 
where $A_{F'}:=A_W|_{F'}$. Thus $|A_{F}|$ defines a birational map where $A_F$ is the 
pushdown of $A_{F'}$. This in turn implies $|M_F|$ defines a birational map because $A_F\le M_F$. 

Moreover,  
$$
K_F+{C}_F+\Theta_F+P_F\sim_\R (K_X+C+\Delta)|_F
$$ 
is nef. In addition,  
$$
M_F-(K_F+{C}_F+S_F+\Theta_F+P_F)\sim_\R (M-(K_X+C+S+\Delta))|_F
$$
is big by the generality of $G$ and by applying Lemma \ref{l-restriction-big-divisors} to the family of subvarieties to which $G$ belongs. This in turn implies 
$$
M_F-(K_F+{C}_F+S_F+\Theta_F)
$$
is big as well as $P_F\ge 0$. 

On the other hand, 
by [\ref{B-compl}, Lemma 3.11], 
$$
\mu_D(\Theta_F+ C_F+S_F+M_F)\ge 1
$$ 
for any component $D$ of $C_F+S_F+M_F$ because  each non-zero coefficient of $C+S+M$ 
is $\ge 1$ by assumption (note that although the latter divisor may not be a $\Q$-divisor but the lemma 
still applies as its proof works for $\R$-divisors as well). 
Similarly, applying the lemma again,
$$
\mu_D(\Theta_F+ \frac{1}{\epsilon} (C_F+S_F))\ge 1
$$
for any component $D$ of $C_F+S_F$ because each non-zero coefficient of $\frac{1}{\epsilon} (C+S)$ 
is $\ge 1$ by assumption. In particular, 
the non-zero coefficients of $\Theta_F+C_F+S_F$ are $\ge \lambda=\min\Psi^{>0}\cup \{\epsilon\}$.

Now applying [\ref{B-compl}, Proposition 4.4] to $F, B_F:=\Theta_F+C_F+S_F, M_F$, 
 there is a projective log smooth couple 
$(\overline{F},{\Sigma}_{\overline{F}})\in \mathcal{P}$ 
and a birational map $\overline{F}\bir F$ satisfying: 
\begin{itemize}
\item  ${\Sigma}_{\overline{F}}$ contains the exceptional 
divisor of  $\overline{F}\bir F$ and the birational transform of $\Supp (\Theta_F+C_F+S_F+M_F)$, and 

\item if $f\colon F'\to F$ and $g\colon F'\to \overline{F}$ is a common resolution and 
$M_{\overline{F}}$ is the pushdown of $M_F|_{F'}$, then each coefficient of $M_{\overline{F}}$ is at most $c$.\\ 
\end{itemize}

\emph{Step 4.} 
In this step we compare log divisors on $F$ and $\overline{F}$.
 First define $ {\Gamma}_{\overline{F}}:=(1-\epsilon) {\Sigma}_{\overline{F}}$.
Let $K_{F'}+\tilde{C}_{F'}$ be the pullback of $K_{F}+\tilde{C}_{F}$ and let $K_{\overline{F}}+\tilde{C}_{\overline{F}}$
be the pushdown of $K_{F'}+\tilde{C}_{F'}$ to $\overline{F}$. 
We claim that $\tilde{C}_{\overline{F}}\le {\Gamma}_{\overline{F}}$. 
If $\tilde{C}_{\overline{F}}\le 0$, then the claim holds trivially. 
Assume  $\tilde{C}_{\overline{F}}$ has a component $D$ with positive coefficient. Then $D$ 
is either exceptional$/F$ or is a component of the  birational transform of $\tilde{C}_{F}$ with positive coefficient. 
In the former case, $D$ is a component of ${\Sigma}_{\overline{F}}$ because  ${\Sigma}_{\overline{F}}$
contains the exceptional divisor of  $\overline{F}\bir F$. In the latter case,  $D$ is a component of 
the birational transform of $\Theta_F+C_F$ because $\tilde{C}_F\le \Theta_F+C_F$ by Step 2, hence again 
$D$ is a component of ${\Sigma}_{\overline{F}}$ as it contains 
the birational transform of 
$
\Supp (\Theta_F+C_F).
$ 
Moreover, since $(F,\tilde{C}_{F})$ is sub-$\epsilon$-lc, 
the coefficient of $D$ in $\tilde{C}_{\overline{F}}$ is at most $1-\epsilon$, hence 
$\mu_D\tilde{C}_{\overline{F}}\le \mu_D{\Gamma}_{\overline{F}}$. We have then proved the claim 
$\tilde{C}_{\overline{F}}\le {\Gamma}_{\overline{F}}$.\\

\emph{Step 5.} 
In this step we define a divisor $I_F$ and compare singularities on $F$ and $\overline{F}$. 
Let 
$$
I_F:=\Theta_F+P_F-\Lambda_F.
$$ 
By Step 2, $I_F\ge 0$.  
First note 
$$
I_{{F}}=\Theta_F+P_F-\Lambda_F=K_F+\Theta_F+P_F-K_F-\Lambda_F 
$$
$$
\sim_\R  (K_X+\Delta)|_F-K_X|_F\sim_\R \Delta|_F.
$$

Pick $0\le L_F\sim_\R M_F$.  
Recalling $\tilde{C}_F=C_F+\Lambda_F$ from Step 2, we see that 
$$
K_F+\tilde{C}_F+I_F+tL_F=K_F+C_F+\Lambda_F+\Theta_F+P_F-\Lambda_F+tL_F
$$
$$
=K_F+C_F+\Theta_F+P_F+tL_F\sim_\R (K_X+C+\Delta+tM)|_F
$$ 
is nef.

Let $I_{\overline{F}}=g_*f^*I_F$, $S_{\overline{F}}=g_*f^*S_F$, and $L_{\overline{F}}=g_*f^*L_F$.
Then by the previous paragraph and by the negativity lemma,
$$
f^*(K_F+\tilde{C}_F+I_F+tL_F)\le g^*(K_{\overline{F}}+\tilde{C}_{\overline{F}}+I_{\overline{F}}+tL_{\overline{F}})
$$
which implies that 
$$
(F,\tilde{C}_F+I_F+tL_F)
$$
is sub-$\epsilon'$-lc if 
 $$
 ({\overline{F}},\tilde{C}_{\overline{F}}+I_{\overline{F}}+tL_{\overline{F}})
 $$ 
 is sub-$\epsilon'$-lc.\\

\emph{Step 6.} 
In this step we finish the proof using [\ref{B-compl}, Proposition 4.2]. 
Note  
$$
2tM_F-(I_{F}+S_F+tL_F)\sim_\R 2tM_F-\Delta|_F-S|_F-tM_F
$$
$$
\sim_\R (tM-\Delta-S)|_F
$$
is big by the generality of $G$ and by Lemma \ref{l-restriction-big-divisors} as $tM-\Delta-S$ 
is big by assumption.
Thus there is  
$$
0\le J_{\overline{F}}\sim_\R 2tM_{\overline{F}}-(I_{\overline{F}}+S_{\overline{F}}+tL_{\overline{F}}).
$$

By construction, 
\begin{itemize}
\item $({\overline{F}},{\Gamma}_{\overline{F}})$ is $\epsilon$-lc, 
\item $({\overline{F}},{\Sigma}_{\overline{F}}=\Supp {\Gamma}_{\overline{F}})\in \mathcal{P}$, 
\item $\Supp M_{\overline{F}}\subseteq {\Sigma}_{\overline{F}}$,  
\item  we have 
$$
2ltM_{\overline{F}}\sim_\R lI_{\overline{F}}+lS_{\overline{F}}+lJ_{\overline{F}}+tlL_{\overline{F}},
$$ 
and 
\item the coefficients of $2ltM_{\overline{F}}$ are bounded from above by $2ltc=\delta$.
\end{itemize}
By our choice of $\delta$ and by applying
[\ref{B-compl}, Proposition 4.2] to 
$$
({\overline{F}},{\Gamma}_{\overline{F}}), ~~~ {\Sigma}_{\overline{F}}, ~~~2ltM_{\overline{F}}, ~~~lI_{\overline{F}}+lS_{\overline{F}}+lJ_{\overline{F}}+tlL_{\overline{F}}
$$
(these are $(X,B), T, N, L$ in the notation of [\ref{B-compl}, Proposition 4.2], respectively)
we deduce that 
$$
({\overline{F}},{\Gamma}_{\overline{F}}+lI_{\overline{F}}+lS_{\overline{F}}+lJ_{\overline{F}}+ltL_{\overline{F}})
$$
is klt, hence 
$$
({\overline{F}},{\Gamma}_{\overline{F}}+lI_{\overline{F}}+ltL_{\overline{F}})
$$
is klt as well. 

Therefore, keeping in mind that $({\overline{F}},{\Gamma}_{\overline{F}})$ is $\epsilon$-lc we see that 
$$
({\overline{F}},{\Gamma}_{\overline{F}}+I_{\overline{F}}+tL_{\overline{F}})
$$ 
is ${\epsilon'}$-lc by [\ref{B-BAB}, Lemma 2.3] because 
$$
{\Gamma}_{\overline{F}}+I_{\overline{F}}+tL_{\overline{F}}
=\left(\frac{l-1}{l}\right){\Gamma}_{\overline{F}}
+\frac{1}{l}({\Gamma}_{\overline{F}}+lI_{\overline{F}}+ltL_{\overline{F}})
$$
and because $(\frac{l-1}{l})\epsilon>\epsilon'$.
This then 
 implies that 
 $$
 ({\overline{F}},\tilde{C}_{\overline{F}}+I_{\overline{F}}+tL_{\overline{F}})
 $$ 
 is sub-$\epsilon'$-lc as $\tilde{C}_{\overline{F}}\le \Gamma_{\overline{F}}$ by Step 4.
Therefore, by Step 5, 
$$
(F,\tilde{C}_F+I_F+tL_F)
$$
is also sub-$\epsilon'$-lc. In other words,
$$
(F,C_F+\Theta_F+P_F+tL_F)
$$ 
is  $\epsilon'$-lc because 
$$
\tilde{C}_F+I_F+tL_F=C_F+\Theta_F+P_F+tL_F
$$
as we saw in Step 5.

\end{proof}

It is clear from the proof that if we replace the nef condition of $K_X+C+\Delta$ with $K_X+C+S+\Delta$ being nef, 
then we can deduce that  
$$
(F,C_F+S_F+\Theta_F+P_F+tL_F)
$$ 
is $\epsilon'$-lc.

\subsection{Descent of divisor coefficients along fibrations}
The adjunction formula in the next lemma is adjunction for fibrations, cf. 
[\ref{kaw-subadjuntion}][\ref{ambro-adj}][\ref{B-compl}, 3.4].

\begin{lem}\label{l-descent-weil-divs-fib}
Let $\epsilon$ be a positive real number. 
Then there is a natural number $l$ depending only on $\epsilon$ satisfying the following. 
Let $(X,B)$ be a klt pair and $f\colon X\to Z$ be a contraction. 
Assume that 
\begin{itemize}
\item $K_X+B\sim_\R 0/Z$,

\item $K_X+B\sim_\R f^*(K_Z+B_Z+M_Z)$ is the adjuction formula for fibrations,

\item $E=f^*L$ for some $\R$-Cartier $\R$-divisor $L$ on $Z$,

\item $\Phi\subset \R$ is a subset closed under multiplication with elements of $\N$,  
 
\item the coefficients of $E$ are in $\Phi$, and 

\item any component $D$ of $L$ has coefficient $\le 1-\epsilon$  in $B_Z$. 
\end{itemize}
Then the coefficients of $lL$ are in $\Phi$.
\end{lem}
\begin{proof}
We can assume $L\neq 0$.   
Let $D$ be a component of $L$ and let $u$ be its coefficient in $L$. 
Shrinking $Z$ around the generic of $D$ we can assume that $D$ is the only component of $L$.
Let $t$ be the lc threshold of $f^*D$ with 
respect to $(X,B)$ over the generic point of $D$. By definition of adjunction for fibrations, 
the coefficient of $D$ in $B_Z$ is $1-t$.
Since $1-t\le 1-\epsilon$ by assumption,  $t\ge \epsilon$. 
So $(X,B+\epsilon f^*D)$ is lc over 
the generic point of $D$. Shrinking $Z$ around the generic point of 
$D$ again we can assume that 
$(X,B+\epsilon f^*D)$ is lc everywhere and that $f^*D$ is an integral divisor. 
Thus the coefficients of $f^*D$ are bounded 
from above by $v:=\lceil \frac{1}{\epsilon} \rceil$. Now by assumption, $E=f^*L=uf^*D$ has coefficients in $\Phi$. 
If $C$ is a component of $E$ with coefficient $e$ and if $h$ is its coefficient in 
$f^*D$, then $e=uh\in \Phi$ where $h\le v$. 
Therefore, letting $l=v!$ we see that $ul\in \Phi$ as $ul$ is a multiple of $e$ and 
$\Phi$ is closed under multiplication with elements of $\N$. 

\end{proof}

\subsection{Descent of divisor coefficients to non-klt centres}
Next we will show that coefficients of divisors restricted to non-klt centres behave well under suitable conditions, 
as in the next proposition.
The first half of the proof of the proposition is similar to that of [\ref{B-compl}, Proposition 3.15] but 
the second half is very different. The proposition is one of the key elements which will allow us to consider 
birational boundedness of divisors in a vastly more general setting than that considered in [\ref{B-compl}] 
which treated only anti-canonical divisors of Fano varieties.

\begin{prop}\label{p-non-klt-centres-restriction-weil-div}
Let $\epsilon$ be a positive real number. 
Then there is a natural number $q$ depending only on 
$\epsilon$ satisfying the following.
Let $(X,B), \Delta, G,F, \Theta_F,P_F$ be as in \ref{constr-adjunction-non-klt-centre}. 
Assume in addition that 
\begin{itemize}
\item $P_F$ is big and for any choice of $P_F\ge 0$ in its $\R$-linear equivalence class 
the pair $(F,\Theta_F+P_F)$ is $\epsilon$-lc, 

\item $\Phi\subset \R$ is a subset closed under addition,  

\item $E$ is an $\R$-Cartier $\R$-divisor on $X$ with coefficients in $\Phi$, and 

\item $G\not\subset \Supp E$.\\
\end{itemize}
Then $qE|_F$ has coefficients in $\Phi$.  

\end{prop}
\begin{proof}
\emph{Step 1}.
 In this step we show $G$ is an isolated non-klt centre of $(X,B+\Delta)$.
We use the notation of \ref{constr-adjunction-non-klt-centre}. 
Remember that $(X,B+\Delta)$ is lc near the generic point of $G$. 
Also recall that $\Gamma+N=B+\Delta$, 
$$
K_Y+\Gamma_Y+N_Y=\pi^*(K_X+\Gamma+N),
$$
and $\rddown{\Gamma_Y}$ has a unique component $S$ mapping onto $G$ 
where $\pi$ denotes $Y\to X$. 

 Assume $G$ is not an isolated non-klt centre. Then some non-klt centre $H\neq G$
of $(X,\Gamma+N)$ intersects $G$. Applying [\ref{B-compl}, Lemma 3.14(2)], we can choose 
$P_F\ge 0$ in its $\R$-linear equivalence class such that 
the pair $(F,\Theta_F+P_F)$ is not $\epsilon$-lc, a contradiction. Therefore, 
$(Y,\Gamma_Y+N_Y)$ is plt near $S$ because  
any non-klt centre other than $S$ would map to a non-klt centre on $X$ disjoint from $G$ because $G$ is an isolated non-klt centre and because $S$ is the unique non-klt place with centre $G$.
In particular, no component of $\rddown{\Gamma_Y}-S$ intersects $S$ which implies that 
any prime exceptional divisor $J\neq S$ of $\pi$ is disjoint from $S$ because all the prime 
exceptional divisors of $\pi$ are components of $\rddown{\Gamma_Y}$.\\

\emph{Step 2.}
In this step we show that there is a 
natural number $p$ depending only on $\epsilon$ such that $pE_S$ has coefficients in $\Phi$ 
where $E_S=E_Y|_S$ and $E_Y=\pi^*E$.
Note that $E_S$ and $E|_F$ are well-defined as $\Q$-Weil divisors as $\Supp E$ does not contain $G$. 
Moreover, $S$ is not a component of $E_Y$ and near $S$, $E_Y$ is just the birational transform of 
$E$ because no exceptional divisor $J\neq S$ of $\pi$ intersects $S$, by Step 1.
Thus the coefficients of $E_Y$ near $S$ belong to $\Phi$. 

Let $V$ be a prime divisor on $S$.  
If $V$ is horizontal over $F$, then $V$ is not a component of $E_S$ because $E_S$ is vertical over $F$. 
Assume $V$ is vertical over $F$. We claim that if $T\neq S$ is any prime divisor on $Y$, then 
 $pT$ is Cartier near the generic point of $V$, 
for some natural number $p$ depending only on $\epsilon$.
Let $l$ be the Cartier index of $K_Y+S$ near the generic point of $V$. 
Then 
$$
\mu_V(\Gamma_S+N_S)\ge 1-\frac{1}{l}
$$ 
and the Cartier index of $T$ near the generic point of $V$ 
divides $l$, by [\ref{Shokurov-log-flips}, Proposition 3.9]. 

Assume $\frac{1}{l}< \epsilon$. Then 
$$
\mu_V(\Gamma_S+N_S)>1-\epsilon,
$$ 
hence $(S,\Gamma_S+N_S)$ is not $\epsilon$-lc along $V$. But then since $V$ is vertical over $F$,  
by [\ref{B-compl}, Lemma 3.14(1)],
we can choose $P_F\ge 0$ so that $(F,\Theta_F+P_F)$ is not $\epsilon$-lc, a contradiction. 
Therefore, $\frac{1}{l}\ge \epsilon$, so $\frac{1}{\epsilon}\ge l$. 
Now let $p=\rddown{\frac{1}{\epsilon}}!$. Then $pT$ is Cartier near the generic point of $V$ as claimed.

Let $E_1,\dots,E_r$ be the irreducible components of $E_Y$ which intersect $S$ and let $e_1,\dots,e_r$ 
be their coefficients. Then 
$$
\mu_V(pE_S)=\mu_V(\sum pe_iE_i|_S)=\sum e_i(\mu_VpE_i|_S).
$$ 
Since $\mu_VpE_i|_S$ 
are positive integers, since $e_i\in \Phi$, and since $\Phi$ is closed under addition, we see that 
 $\mu_V(pE_S)\in \Phi$.\\
 
\emph{Step 3}.
In this step we consider adjunction for $(S,\Gamma_S+N_S)$ over $F$. 
Recall that $S\to F$ is denoted by $h$ which is a contraction. 
Since $(Y,\Gamma_Y+N_Y)$ is plt near $S$, $(S,\Gamma_S+N_S)$ is klt. 
As $K_S+\Gamma_S+N_S\sim_\R 0/F$, we have the adjunction 
formula
$$
K_S+\Gamma_S+N_S\sim_\R h^*(K_F+\Omega_F+M_F)
$$
where $\Omega_F$ is the discriminant divisor and $M_F$ is the moduli divisor which is 
pseudo-effective (cf. [\ref{B-compl}, 3.4]). 
Recall from \ref{constr-adjunction-non-klt-centre} that $\Theta_F$ is defined by taking lc 
thresholds with respect to some boundary $\Sigma_S\le \Gamma_S+N_S$. Then $\Theta_F\le \Omega_F$ 
by definition of adjunction. 

We claim that the coefficients of $\Omega_F$ do not exceed $1-\epsilon$. Assume not. 
Pick a small real number $t>0$ such that some coefficient of 
$$
\Xi_F:=\Theta_F+(1-t)(\Omega_F-\Theta_F)=(1-t)\Omega_F+t\Theta_F
$$ 
exceeds $1-\epsilon$.  By construction, 
$$
K_F+\Omega_F+M_F\sim_\R (K_X+\Gamma+N)|_F=(K_X+B+\Delta)|_F\sim_\R K_F+\Theta_F+P_F,
$$ 
hence 
$$
\Omega_F+M_F\sim_\R \Theta_F+P_F.
$$ 
Since $M_F$ is pseudo-effective and $P_F$ is big, we can find 
$$
0\le R_F\sim_\R (1-t)M_F+tP_F.
$$
Then 
$$
\Xi_F+R_F\sim_\R (1-t)\Omega_F+t\Theta_F+(1-t)M_F+tP_F
$$
$$
=(1-t)(\Omega_F+M_F)+t(\Theta_F+P_F)\sim_\R \Theta_F+P_F
$$
which in particular means $K_F+\Xi_F+R_F$ is $\R$-Cartier. 

Now $(F,\Xi_F+R_F)$ is not $\epsilon$-lc by our choice of $\Xi_F$. 
On the other hand, since 
$$
\Xi_F-\Theta_F= (1-t)(\Omega_F-\Theta_F)\ge 0,
$$ 
choosing 
$$
0 \le P_F:=\Xi_F-\Theta_F+R_F
$$
we see that
$$
(F,\Theta_F+P_F)=(F,\Xi_F+R_F)
$$
is not $\epsilon$-lc, a contradiction. This proves the claim that 
 the coefficients of $\Omega_F$ do not exceed $1-\epsilon$.\\

\emph{Step 4.}
In this step we finish the proof. By Step 2, the coefficients of $pE_S$ are in $\Phi$. 
Moreover, by construction, 
$$
pE_S=h^*(pE|_F).
$$ 
Also since $\Phi$ is closed under addition, it is closed under multiplication with 
elements of $\N$. Thus applying Lemma \ref{l-descent-weil-divs-fib} to $(S,\Gamma_S+N_S)\to F$ and $pE_S$, 
there is a natural number $l$ depending only on 
$\epsilon$ such that the coefficients of $lpE|_F$ are in $\Phi$. Now let $q=lp$.

\end{proof}

In practice examples of $\Phi$ include $\Phi=\Z$ and $\Phi=\{r\in \R \mid r\ge \delta\}$ 
for some $\delta$.


\section{\bf Birational boundedness on $\epsilon$-lc varieties}

In this section we treat birational boundedness of linear systems of nef and big  
divisors on $\epsilon$-lc varieties which serves as the basis for the subsequent sections. 

\subsection{Main result}
The following theorem is a more general form of \ref{t-eff-bir-nefbig-weil} and the main result of this section.
 
\begin{thm}\label{t-eff-bir-general}
Let $d$ be a natural number and $\epsilon,\delta$ be positive real numbers. 
Then there exists a natural number $m$ depending only on $d,\epsilon,\delta$ satisfying the following. Assume 
\begin{itemize}
\item $X$ is a projective $\epsilon$-lc variety of dimension $d$,

\item $N$ is a nef and big $\R$-divisor on $X$, 

\item $N-K_X$ is pseudo-effective, and 

\item $N=E+R$ where $E$ is integral and pesudo-effective and $R\ge 0$ with coefficients in $\{0\}\cup [\delta,\infty)$.

\end{itemize}
\vspace{0.2cm}
Then $|m'N+L|$ and $|K_X+m'N+L|$ define birational maps for any integral pseudo-effective divisor $L$ 
and for any natural number $m'\ge m$.
\end{thm}

Special cases of the theorem are when $R=0$ which is the statement of \ref{t-eff-bir-nefbig-weil}, and 
when $E=0$ in which case $N$ is effective with coefficients $\ge \delta$. The theorem does not hold if 
we drop the pseudo-effectivity condition of $E$: indeed, one can easily find counter-examples 
by considering $\epsilon$-lc Fano varieties $X$, $\Q$-divisors $0\le B\sim_\Q -K_X$ with coefficients $\ge \delta$ 
and then letting $E=K_X$ and $R=B+tB$ with $t$ arbitrarily small.

The theorem implies a more general form of \ref{cor-eff-bir-nefbig-weil}.

\begin{cor}\label{cor-eff-bir-nefbig-general}
Let $d$ be a natural number and $\epsilon,\delta$ be positive real numbers. 
Then there exist natural numbers $m,l$ depending only on $d,\epsilon, \delta$ satisfying the following. Assume that 
\begin{itemize}
\item $X$ is a projective $\epsilon$-lc variety of dimension $d$, 

\item $N$ is a nef and big $\R$-divisor on $X$, and

\item $N=E+R$ where $E$ is integral and pseudo-effective and $R\ge 0$ with coefficients in 
$\{0\}\cup [\delta,\infty)$.
\end{itemize}
\vspace{0.2cm}
Then $|m'K_X+l'N+L|$ defines a birational map for any natural numbers $m'\ge m$ and $l'\ge m'l$ and any
pseudo-effective integral divisor $L$.
 
\end{cor}\

\subsection{Birational or volume boundedness}
In this and the next subsection
we aim to derive a special case of Theorem \ref{t-eff-bir-general}, where we add the condition that $K_X+N$ is also big, 
from the theorem in lower dimension. Later in the section 
we will remove this condition using the BAB [\ref{B-BAB}, Theorem 1.1] 
in dimension $d$.
We begin with a birationality statement in which we either bound the birationality index or the volume of the 
relevant divisor.

\begin{prop}\label{p-eff-bir-nefbig-weil-1}
Let $d$ be a natural number and $\epsilon,\delta$ be positive real numbers. Assume that 
Theorem \ref{t-eff-bir-general} holds in dimension $\le d-1$.
Then there exists a natural number $v$ depending only on $d,\epsilon,\delta$ satisfying the following. Assume 
\begin{itemize}
\item $X$ is a projective $\epsilon$-lc variety of dimension $d$,

\item $N$ is a nef and big $\Q$-divisor on $X$,

\item $N-K_X$ and $N+K_X$ are big, and

\item $N=E+R$ where $E$ is integral and pseudo-effective and $R\ge 0$ with coefficients in 
$\{0\}\cup [\delta,\infty)$.
\end{itemize}
\vspace{0.2cm}
If $m$ is the smallest natural number such that $|mN|$ defines a birational map, then 
 either $m\le v$ or $\vol(mN)\le v$.
\end{prop}

\begin{proof}
We follow the proof of [\ref{B-compl}, Proposition 4.8].

\emph{Step 1}.
In this step we setup basic notation and make some reductions. 
If the proposition does not hold, then there is a sequence $X_i,N_i$ of varieties 
and divisors as in the proposition such that if $m_i$ is the smallest natural number 
so that $|m_iN_i|$ defines a birational map, then both the numbers $m_i$ and 
the volumes $\vol(m_iN_i)$ form increasing sequences approaching $\infty$. 
Let $n_i\in\N$ be a natural number so that $\vol(n_iN_i)>(2d)^d$. 
Taking a $\Q$-factorialisation we can assume $X_i$ are $\Q$-factorial. 
 
Assume that $\frac{m_i}{n_i}$ is always (i.e. for any choice 
of $n_i$ as above) bounded from above. 
Then letting $n_i\in\N$ be the smallest number so that $\vol(n_iN_i)>(2d)^d$, 
 either $n_i=1$ which immediately implies $m_i$ is bounded from above, or $n_i>1$ in which case 
we have 
$$
\vol(m_iN_i)=(\frac{m_i}{n_i-1})^d \vol((n_i-1)N_i)\le (\frac{m_i}{n_i-1})^d(2d)^d
$$
so $\vol(m_iN_i)$ is bounded from above. 

Therefore, it is enough to show that $\frac{m_i}{n_i}$ is always bounded from above 
where $n_i$ is arbitrary as in the first paragraph. Assume otherwise. 
We can then assume that the numbers $\frac{m_i}{n_i}$ form an increasing sequence 
approaching $\infty$. We will derive a contradiction.\\

\emph{Step 2.}
In this step we fix $i$ and create a covering family of non-klt centres on $X_i$. 
Replacing $n_i$ with $n_i+1$ we can assume $\vol((n_i-1)N_i)>(2d)^d$. Thus 
applying \ref{ss-non-klt-centres}(2),   
there is a  covering family of subvarieties of $X_i$ 
such that for any two general closed points $x_i,y_i\in X_i$ we can choose a  member  
 $G_i$ of the family and choose a $\Q$-divisor 
$$
0\le \tilde{\Delta}_i\sim_\Q (n_i-1)N_i
$$  
so that $(X_i,\tilde{\Delta}_i)$ is lc near $x_i$ with a unique 
non-klt place whose centre contains $x_i$, that centre is $G_i$, and $(X_i,\tilde{\Delta}_i)$ is not klt near $y_i$.

On the other hand, by assumption $N_i-K_{X_i}$ is big. Fix some 
$$
0\le Q_i\sim_\Q N_i-K_{X_i}, 
$$
independent of the choice of $x_i,y_i$; we can assume $x_i,y_i$ are not contained in $\Supp Q_i$ as 
$x_i,y_i$ are general. 
Then  we get 
$$
0\le \Delta_i:=\tilde{\Delta}_i+Q_i\sim_\Q n_iN_i-K_{X_i} 
$$
such that $(X_i,\Delta_i)$ is lc near $x_i$ with a unique 
non-klt place whose centre contains $x_i$, that centre is $G_i$, and $(X_i,{\Delta}_i)$ is not klt near $y_i$.

Since $x_i,y_i$ are general, we can assume $G_i$ is a general member of 
the above covering family of subvarieties. 
Recall from \ref{ss-cov-fam-subvar} that this means that the family is given by finitely many 
morphisms $V^j\to T^j$ of projective varieties with accompanying surjective morphisms $V^j\to X$ 
and that each $G_i$ is a general fibre of one of the morphisms $V^j\to T^j$. Moreover, we can assume 
the points of $T^j$ corresponding to such $G_i$ are dense in $T^j$. 

Let 
$$
d_i:=\max_j\{\dim V^j-\dim T^j\}.
$$
Assume $d_i=0$, that is, $\dim G_i=0$ for all the $G_i$. Then $n_iN_i-K_{X_i}$ is 
potentially birational. Let $r\in \N$ be the smallest number such that $r\delta\ge 1$. 
By assumption, $N_i=E_i+R_i$ with $E_i$ integral and pseudo-effective and $R_i\ge 0$ whose non-zero 
coefficients are $\ge \delta$. Thus the fractional part of $(n_i+r)N_i-K_{X_i}$, say $R_i'$, 
is supported in $R_i$ and the coefficients of $rR_i$ are $\ge 1$. So 
$$
\rddown{(n_i+r)N_i-K_{X_i}}=(n_i+r)N_i-K_{X_i}-R_i'={n_i}N_i-K_{X_i}+rE_i+rR_i-R_i'
$$ 
where $rR_i-R_i'\ge 0$, hence $\rddown{(n_i+r)N_i-K_{X_i}}$ is potentially birational.  
Therefore,  
$$
|\rddown{(n_i+r)N_i}|=|K_{X_i}+\rddown{(n_i+r)N_i-K_{X_i}}|
$$ 
defines a birational map by [\ref{HMX1}, Lemma 2.3.4] which in turn implies $|(n_i+r)N_i|$ 
defines a birational map; this means $m_i\le n_i+r$ giving a contradiction as 
we can assume $m_i/n_i\gg 0$. Thus we can assume $d_i>0$, hence $\dim G_i>0$ for all the $G_i$  
appearing as general fibres of $V^j\to T^j$ for some $j$.\\
 
\emph{Step 3.} 
In this step we find a sub-family of the $G_i$ so that $\vol(m_iN_i|_{G_i})$ is bounded from above, 
independent of $i$. For each $i$ let $l_i\in\N$ be the smallest number so that $\vol(l_iN_i|_{G_i})>d^d$ for all 
the $G_i$ with positive dimension. 
Assume $\frac{l_i}{n_i}$ is bounded from above by some natural number $a$. Then for each $i$ and 
each positive dimensional $G_i$, we have 
$$
{d^d}< \vol(l_iN_i|_{G_i}) \le \vol(an_iN_i|_{G_i}). 
$$
Thus applying Lemma \ref{l-non-klt-centres-ii} and replacing $n_i$ with $3an_i$  
we can modify $\Delta_i,G_i$, and possibly switch $x_i,y_i$, so that we decrease the number $d_i$.  
Repeating the process we either get to the situation in which $d_i=0$ which yields a contradiction 
as in Step 2, or we can assume $\frac{l_i}{n_i}$ is an 
increasing sequence approaching $\infty$. 

On the other hand, if $\frac{m_i}{l_i}$ is not bounded from above, then we can assume 
$\frac{m_i}{l_i}$ is an increasing sequence approaching $\infty$, hence 
we can replace $n_i$ with $l_i$ (by adding appropriately to $\Delta_i$) in which case $\frac{l_i}{n_i}$ is bounded so we 
can argue as in the previous paragraph by decreasing $d_i$. So we can assume $\frac{m_i}{l_i}$ is bounded from above.

Now for each $i$, there is $j$ so that if $G_i$ is a general fibre of 
$V^j\to T^j$, then $G_i$ is positive dimensional and 
$$
\vol((l_i-1)N_i|_{G_i})\le d^d,
$$ 
 by definition of $l_i$.
In order to get a contradiction in the following steps it suffices to consider only such $G_i$.
\emph{From now on when we mention $G_i$ we assume it is positive dimensional and that it 
satisfies the  inequality just stated}.
  In particular,  
$$
\vol(m_iN_i|_{G_i})=\left(\frac{m_i}{l_i-1}\right)^{\dim G_i}\vol((l_i-1)N_i|_{G_i})\le \left(\frac{m_i}{l_i-1}\right)^dd^d
$$ 
is bounded from above, so $\vol(m_iN_i|_{G_i})<u$ for some natural number $u$ independent of $i$.\\

\emph{Step 4.} 
Let $F_i$ be the normalisation of $G_i$. In this step we fix $i$ and apply adjunction by restricting to $F_i$.
Since $G_i$ is a general member of a covering family, $X_i$ is smooth  near the generic point 
of $G_i$. By the adjunction of \ref{constr-adjunction-non-klt-centre} (taking $B=0$ and $\Delta=\Delta_i$), we can write 
$$
K_{F_i}+\Delta_{F_i}:=K_{F_i}+\Theta_{F_i}+P_{F_i}\sim_\R (K_{X_i}+\Delta_i)|_{F_i}\sim_\R n_iN_i|_{F_i}
$$
where $\Theta_{F_i}\ge 0$ with coefficients in some fixed DCC set $\Psi$ independent of $i$, 
and $P_{F_i}$ is pseudo-effective. 
Since $x_i$ is a general point, 
we can pick  $0\le \tilde{N}_i\sim_\Q N_i$ not containing $x_i$. By definition of $\Theta_{F_i}$, 
adding $\tilde{N}_i$ to $\Delta_i$ does not change $\Theta_{F_i}$ but changes $P_{F_i}$ to 
$P_{F_i}+\tilde{N}_i|_{F_i}$. Thus replacing 
$n_i$ with $n_i+1$ and changing $P_{F_i}$ up to $\R$-linear equivalence we can assume 
$P_{F_i}$ is effective and big (alternatively, using the notation of Step 2, we can first apply the adjunction to $(X_i,\tilde{\Delta}_i)$ and then add ${Q}_i|_{F_i}$ to deduce $P_{F_i}$ is big).

On the other hand, by [\ref{HMX2}, Theorem 4.2] and [\ref{B-compl}, Lemma 3.12],
 we can write 
$$
K_{F_i}+\Lambda_{F_i}=K_{X_i}|_{F_i}
$$ 
where $(F_i,\Lambda_{F_i})$ is sub-$\epsilon$-lc and $\Lambda_{F_i}\le \Theta_{F_i}$.\\ 

\emph{Step 5.}
In this step we reduce to the situation in which $(F_i,\Delta_{F_i})$ is ${\epsilon}'$-lc 
for every $i$ where $\epsilon'=\frac{\epsilon}{2}$. 
 Pick $0\le M_i\sim m_iN_i$, independent of $x_i,y_i$. Since $G_i$ is general, we can assume 
 $\Supp M_i$ does not contain $G_i$. Let $M_{F_i}:=M_i|_{F_i}$.
Let $t$ be the number given by Proposition \ref{p-bnd-sing-on-non-klt-centre} 
for the data $d,u,\epsilon,\epsilon'$. Let $r\in \N$ be the smallest number such that $r\ge \frac{1}{\delta}$.
We can assume that $\frac{n_i+1+r}{m_i}<t$ for every $i$. 
Let $C_i=0$ and let $S_i=rR_i$. We want to apply \ref{p-bnd-sing-on-non-klt-centre} to 
$$
X_i,C_i,M_i,S_i,\Delta_i,G_i,F_i,\Theta_{F_i},P_{F_i}
$$ 
where $P_{F_i}$ can be any effective divisor in its $\R$-linear equivalence class. 
 
Conditions (1)-(4) of \ref{p-bnd-sing-on-non-klt-centre} are obviously satisfied. Condition (5) is satisfied because 
the non-zero coefficients of $S_i$ are $\ge 1$ and similarly the non-zero coefficients of $M_i+S_i$ are 
$\ge 1$ as the fractional part of $M_i$ is supported in $\Supp R_i=\Supp S_i$; the latter claim follows from  
$$
M_i\sim m_iN_i=m_iE_i+m_iR_i
$$
 and the assumption that $E_i$ is integral. Conditions (6)-(9) are satisfied by construction, and condition 
(10) is ensured by the end of Step 3. Condition (11) follows from $K_{X_i}+\Delta_i\sim_\Q n_iN_i$, and 
(12) follows from bigness of 
$$
M_i-(K_{X_i}+S_i+\Delta_i)\sim_\Q M_i-(K_{X_i}+\Delta_i)-S_i
$$
$$
\sim_\Q m_iN_i-n_iN_i-S_i=(m_i-n_i)N_i-rR_i=(m_i-n_i-r)N_i+rE_i
$$
as we can assume $m_i> n_i+r$. Finally, condition (13) is satisfied because 
from $\frac{n_i+1+r}{m_i}<t$, we have 
$$
t m_i-n_i-1-r>0,
$$ 
hence  
$$
t M_i-\Delta_i-S_i\sim_\R t m_iN_i-n_iN_i+K_{X_i}-S_i
$$
$$
= 
(t m_i-n_i-1-r)N_i+N_i+K_{X_i}+rN_i-S_i
$$
is big as $N_i+K_{X_i}$ is big by assumption and $rN_i-S_i=rE_i$ is pseudo-effective. 
 
Now by Proposition \ref{p-bnd-sing-on-non-klt-centre}, 
 $(F_i,\Delta_{F_i}+t M_{F_i})$ is $\epsilon'$-lc for every $i$ which in particular means 
  $(F_i,\Delta_{F_i})$ is $\epsilon'$-lc.\\ 

\emph{Step 6.}
In this step we finish the proof. Let $E_{F_i}=E_i|_{F_i}$, $R_{F_i}=R_i|_{F_i}$, and 
$N_{F_i}=N_i|_{F_i}$ which are all well-defined as $\Q$-Weil 
divisors as we can assume $G_i$ is not contained in $\Supp E_i\cup \Supp R_i$.
Applying Proposition \ref{p-non-klt-centres-restriction-weil-div}, by taking $\Phi=\Z$, 
there is a natural number $q$ depending only on $\epsilon'$ such that $qE_{F_i}$ is an integral divisor. 
Applying the proposition again, this time taking $\Phi=\R^{\ge \delta}$, and replacing $q$ we can assume that 
the non-zero coefficients of $qR_{F_i}$ are $\ge \delta$. Thus $qN_{F_i}=qE_{F_i}+qR_{F_i}$ is the 
sum of a pseudo-effective integral divisor and an effective $\Q$-divisor whose non-zero coefficients are $\ge \delta$.  
  
On the other hand, by construction, 
$$
qn_iN_{F_i}-K_{F_i}=(q-1)n_iN_{F_i}+n_iN_{F_i}-K_{F_i}\sim_\R (q-1)n_iN_{F_i}+\Delta_{F_i}
$$
is big. Thus since we are assuming Theorem \ref{t-eff-bir-general} in dimension $\le d-1$, 
we deduce that there is a natural number $p$ depending only on $d,\epsilon',\delta$ 
such that $|pqn_iN_{F_i}|$ defines a birational map (note that $K_{F_i}$ may not be 
$\Q$-Cartier but we can apply the theorem to a small $\Q$-factorialisation of $F_i$). Then  
$\vol(pqn_iN_{F_i})\ge 1$, hence 
$$
\vol((2d)^dpqn_iN_i|_{G_i})=\vol((2d)^dpqn_iN_{F_i})\ge (2d)^d.
$$ 
But then by Step 3, we have 
$l_i-1< (2d)^dpqn_i$ which is a contradiction since we assumed 
$\frac{l_i}{n_i}$ is an unbounded sequence.

\end{proof}

\subsection{Birational boundedness}
In this subsection we strengthen \ref{p-eff-bir-nefbig-weil-1} by showing that we can actually 
take $m$ itself to be bounded.

\begin{lem}\label{l-bnd-lc-model-ep-lc}
Let $d,v$ be natural numbers, $\epsilon$ be a positive real number, and $\Phi\subset [0,1]$ be a finite set of 
rational numbers. Assume Theorem \ref{t-eff-bir-general} holds in dimension $\le d-1$.
Assume that 
\begin{itemize}
\item $(X,B)$ is a projective $\epsilon$-lc pair of dimension $d$,

\item the coefficients of $B$ are in $\Phi$,

\item $K_X+B$ is ample with $\vol(K_X+B)\le v$, and 

\item $2K_X+B$ is big.
\end{itemize}
Then such $(X,\Supp B)$ form a bounded family.
\end{lem}
\begin{proof}
There is a natural number $l>1$ depending only on $\Phi$ such that $N:=l(K_X+B)$ is integral. 
Moreover, 
$$
N-K_X=(l-1)(K_X+B)+B
$$ 
and 
$$
N+K_X=(l-1)(K_X+B)+2K_X+B
$$ 
are both big. Now applying Proposition \ref{p-eff-bir-nefbig-weil-1} to a $\Q$-factorialisation of $X$ 
we deduce that there is a natural number $v'$ depending only on $d,\epsilon$ such that if 
$m\in \N$ is the smallest number such that $|mN|$ defines a birational map, then either $m\le v'$ or $\vol(mN)\le v'$. 
If $m\le v'$, then 
$$
\vol(mN)=\vol(ml(K_X+B))\le (ml)^dv\le (v'l)^dv
$$ 
So in any case $\vol(mN)$ is bounded from above. 

Pick $0\le M\sim mN$. Then, by [\ref{B-compl}, Proposition 4.4], 
$(X,\Supp (B+M))$ is birationally bounded. Finally apply [\ref{HMX2}, Theorem 1.6].

\end{proof}

\begin{prop}\label{p-eff-bir-nefbig-weil-2}
Let $d$ be a natural number and $\epsilon,\delta$ be positive real numbers. Assume that 
Theorem \ref{t-eff-bir-general} holds in dimension $\le d-1$.
Then there exists a natural number $m$ depending only on $d,\epsilon,\delta$ satisfying the following. Assume 
\begin{itemize}
\item $X$ is a projective $\epsilon$-lc variety of dimension $d$,

\item $N$ is a nef and big $\Q$-divisor on $X$,

\item $N-K_X$ and $N+K_X$ are big, and

\item $N=E+R$ where $E$ is integral and pseudo-effective and $R\ge 0$ with coefficients in 
$\{0\}\cup [\delta,\infty)$.
\end{itemize}
\vspace{0.2cm}
Then  $|mN|$ defines a birational map.
\end{prop}

\begin{proof}
\emph{Step 1.}
In this step we apply \ref{p-eff-bir-nefbig-weil-1} and introduce some notation. 
Taking a small $\Q$-factorialisation, we will assume $X$ is $\Q$-factorial.
Let $v\in \N$ be the number given by Proposition \ref{p-eff-bir-nefbig-weil-1} for the data 
$d,\epsilon,\delta$. Then there is $m\in \N$ such that $|mN|$ 
defines a birational map and either $m\le v$ or $\vol(mN)\le v$. 
In the former case we are done by replacing $m$ with $v!$, hence we can assume the latter holds 
and that $m$ is sufficiently large. 
 
Since $N-K_X$ is big, we can find 
$$
0\le \Delta\sim_\Q N-K_X.
$$
Pick $0\le M\sim mN$.\\

\emph{Step 2.}
In this step we show that $(X,\Supp (M+R))$ is birationally bounded. 
Let $B=\frac{1}{\delta}R$. Then 
$$
M-(K_X+B)\sim mN-K_X-\frac{1}{\delta}R=(m-\frac{1}{\delta}-1)N+N-K_X+\frac{1}{\delta}N-\frac{1}{\delta}R
$$
is big because $m-\frac{1}{\delta}-1\ge 0$, $N-K_X$ is big, and 
$\frac{1}{\delta}N-\frac{1}{\delta}R=\frac{1}{\delta}E$ is pseudo-effective. 
On the other hand, since $M\sim mE+mR$, 
the fractional part of $M$ is supported in $B$, so for any component $D$ of $M$ 
we have $\mu_D(B+M)\ge 1$. Therefore, by [\ref{B-compl}, Proposition 4.4],  
there exist a bounded set of couples $\mathcal{P}$ and a natural number $c$ 
depending only on $d,v$ such that 
we can find a projective log smooth couple
 $({\overline{X}},\overline{{\Sigma}})\in\mathcal{P}$ and a birational map $\overline{X}\bir X$ such that 
\begin{itemize}
\item  $\Supp \overline{{\Sigma}}$ contains the exceptional 
divisors of  $\overline{X}\bir X$ and the birational transform of $\Supp (B+M)$;

\item if $\phi\colon X'\to X$ and $\psi\colon X'\to \overline{X}$ is a common resolution and 
$\overline{M}$ is the pushdown of $M':=M|_{X'}$, then each coefficient of $\overline{M}$ 
is at most $c$;

\item we can choose $X'$ so that $M'\sim A'+R'$ where $A'$ is big, $|A'|$ is base point free, $R'\ge 0$, 
and $A'\sim 0/\overline{X}$.\\
\end{itemize}

\emph{Step 3.}
In this step we want to show that
there is a positive rational number $t$ depending only on $\mathcal{P},c,\epsilon$ such that we can assume 
$(X,tM)$ is $\frac{\epsilon}{2}$-lc. We do this by finding 
$t$ so that $(X,\Delta+tM)$ is $\frac{\epsilon}{2}$-lc. 
For now let $t$ be a positive rational number.
Then 
$$
K_X+\Delta+tM\sim_\Q N+tM \sim_\Q (1+tm)N
$$ 
is nef and big. 
Let $K_{\overline{X}}+\overline{\Lambda}=\psi_*\phi^*K_X$ be the crepant 
pullback of $K_X$ to $\overline{X}$. Since $X$ has $\epsilon$-lc singularities, the coefficients 
of $\overline{\Lambda}$ are at most $1-\epsilon$. Moreover, $\Supp \overline{\Lambda}\subseteq \overline{\Sigma}$. 
On the other hand, since $N+K_X$ is assumed big, we can find 
$$
0\le C\sim_\Q N+K_X.
$$
Note that $\Delta+C\sim_\Q 2N$.
Then 
$$
(\frac{2}{m}+t)M\sim_\Q 2N+tM\sim_\Q \Delta+C+tM.
$$ 

Now if $\overline{\Delta}=\psi_*\phi^*\Delta$ and $\overline{C}=\psi_*\phi^*C$, 
then  
$$
\psi_*\phi^*2(\Delta+C+tM)=2(\overline{\Delta}+\overline{C}+t\overline{M}) \sim_\Q (\frac{4}{m}+2t)\overline{M}.
$$
By Step 2, $\Supp \overline{M}\subseteq \overline{\Sigma}$ and  
 $(\frac{4}{m}+2t)\overline{M}$ has coefficients $\le (\frac{4}{m}+2t)c$. Therefore, 
applying [\ref{B-compl}, Proposition 4.2], by taking   
$m$ sufficiently large and taking a fixed $t$ sufficiently small, depending only on $\mathcal{P},c,\epsilon$, 
we can assume that 
$$
(\overline{X},(1-\epsilon)\overline{\Sigma}+2\overline{\Delta}+2\overline{C}+2t\overline{M})
$$ 
is klt. Then 
$$
(\overline{X},(1-\epsilon)\overline{\Sigma}+\overline{\Delta}+\overline{C}+t\overline{M})
$$ 
is $\frac{\epsilon}{2}$-lc as $(\overline{X},(1-\epsilon)\overline{\Sigma})$ is $\epsilon$-lc which in turn implies 
$$
(\overline{X},\overline{\Lambda}+\overline{\Delta}+t\overline{M})
$$ 
is sub-$\frac{\epsilon}{2}$-lc because $\overline{\Lambda}\le (1-\epsilon)\overline{\Sigma}$.

Now since $K_X+\Delta+tM$ is nef and since its crepant pullback to $\overline{X}$ is just 
$$
K_{\overline{X}}+\overline{\Lambda}+\overline{\Delta}+t\overline{M},
$$ 
we deduce that $(X,\Delta+tM)$ is $\frac{\epsilon}{2}$-lc  
because by the negativity lemma
$$
\phi^*(K_X+\Delta+tM)\le \psi^*(K_{\overline{X}}+\overline{\Lambda}+\overline{\Delta}+t\overline{M}). 
$$
Therefore, $(X,tM)$ is $\frac{\epsilon}{2}$-lc.\\ 

\emph{Step 4.}
In this step we finish the proof by applying Lemma \ref{l-bnd-lc-model-ep-lc}.
First we show $K_X+t\rddown{M}$ and $2K_X+t\rddown{M}$ are big.
As noted above, the factional part of $M$, say $P$, is supported in $R$. 
Thus $\frac{1}{\delta}R-P\ge 0$. 
Then since $m$ is sufficiently large and $t$ is fixed, and since $K_X+N$ is big, 
we can ensure that 
$$
K_X+t\rddown{M}= K_X+tM-tP\sim_\Q K_X+tmN-tP 
$$
$$
= K_X+N+(tm-1-\frac{t}{\delta})N+\frac{t}{\delta}E+\frac{t}{\delta}R-tP
$$
is big. Similar reasoning shows that $2K_X+t\rddown{M}$ is also big. 

Now $(X,t\rddown{M})$ has an lc model, say $(Y,t\rddown{M}_Y)$. By construction, $(Y,t\rddown{M}_Y)$ is 
$\frac{\epsilon}{2}$-lc, the coefficients of $t\rddown{M}_Y$ are in a fixed finite set depending only 
on $t$, and 
$$
\vol(K_Y+t\rddown{M}_Y)= \vol(K_X+t\rddown{M})< \vol(N+tM)=\vol((\frac{1}{m}+t)M)
$$ 
is bounded from above as $\vol(M)\le v$ by Step 1. 
Also $2K_Y+t\rddown{M}_Y$ is big. Therefore, by Lemma \ref{l-bnd-lc-model-ep-lc}, 
such $(Y,\Supp \rddown{M}_Y)$ form a bounded family. In particular, there is a very ample divisor 
$A_Y$ on $Y$ such that $A_Y^{d}$ and $A_Y^{d-1}\cdot \Supp \rddown{M}_Y$ are bounded from above (note that the intersection 
number $A_Y^{d-1}\cdot Q$ is well-defined for any $\R$-divisor even if $Q$ is not $\R$-Cartier). 
Since the coefficients of $\rddown{M}_Y$ are bounded from above, $A_Y^{d-1}\cdot \rddown{M}_Y$ is also 
bounded from above.

Now since 
$$
\rddown{M}_Y= mN_Y-P_Y \sim_\Q (m-\frac{1}{\delta})N_Y+\frac{1}{\delta}E_Y+\frac{1}{\delta}R_Y-P_Y,
$$
where $\frac{1}{\delta}E_Y+\frac{1}{\delta}R_Y-P_Y$ is pseudo-effective, we see that 
$A_Y^{d-1}\cdot (m-\frac{1}{\delta})N_Y$ is bounded from above. However, 
$$
A_Y^{d-1} \cdot N_Y= A_Y^{d-1} \cdot (E_Y+R_Y)\ge \delta.
$$ 
Therefore,   
$$
A_Y^{d-1}\cdot (m-\frac{1}{\delta})N_Y\ge (m-\frac{1}{\delta})\delta 
$$ 
can get arbitrarily large, a contradiction.  

\end{proof}

\subsection{Some remarks}\label{rem-Fano-eff-bir-new-proof}
(1) {So far in this and the previous sections we have tried to avoid using 
results of [\ref{HMX1}][\ref{HMX2}][\ref{B-compl}][\ref{B-BAB}] 
as much as possible because one of our goals (here and in the future) is to get new proofs of some of the 
main results of those papers. To be more precise, when trying to prove a statement in dimension $d$ 
we have tried to minimise relying on results of those papers in dimension $d$. 
Although we have used many of the technical auxiliary results and ideas of the first three 
papers but we have not used their main results in dimension $d$ with the exception of 
 [\ref{HMX1}, Theorem 1.8]: indeed 
we used  [\ref{HMX2}, Theorem 1.6] in the proof of Lemma \ref{l-bnd-lc-model-ep-lc}, 
which is a quick consequence of [\ref{HMX1}, Theorem 1.8]. 
 
(2) 
We will argue that if in Propositions \ref{p-eff-bir-nefbig-weil-1} and \ref{p-eff-bir-nefbig-weil-2} 
we assume $R\neq 0$, then we can modify their proofs so that we do not need Lemma \ref{l-bnd-lc-model-ep-lc} 
nor Theorem \ref{t-eff-bir-general} in lower dimension. 
Indeed we can modify Step 6 of the proof of \ref{p-eff-bir-nefbig-weil-1} as follows. Since $R_i$ has non-zero 
coefficients $\ge \delta$, adding $\frac{1}{\delta}N$ to $\Delta$ and using [\ref{B-compl}, Lemma 3.14(2)] 
we can ensure that $(F_i,\Delta_{F_i})$ is not $\epsilon'$-lc, contradicting Step 5.

In \ref{p-eff-bir-nefbig-weil-2}, we can modify Steps 3-4 of the proof as follows. By the previous paragraph, 
we can find $m$ such that $|mN|$ defines a birational map 
and either $m$ or $\vol(mN)$ is bounded from above. Since $R\neq 0$ and since its coefficients are $\ge \delta$, 
there is 
$$
0\le L\sim_\Q \frac{3}{\delta} N=\frac{1}{\delta}N+\frac{2}{\delta}E+\frac{2}{\delta}R
$$ 
such that some coefficient of $L$ exceeds $1$. On the other hand,
$$
L+\Delta+C+tM\sim_\Q (\frac{3}{m\delta}+\frac{3}{m}+t)M,
$$
so applying the arguments of Step 3 we can assume $(X,L+tM)$ is $\frac{\epsilon}{2}$-lc. 
This is a contradiction because $(X,L+tM)$ is not lc as some coefficient of $L$ 
exceeds $1$.

(3) Let $X$ be an $\epsilon$-lc Fano variety of dimension $d$, for $\epsilon>0$, and let $N=-2K_X$.
Assuming Theorem \ref{t-eff-bir-general} in lower dimension and 
applying Proposition \ref{p-eff-bir-nefbig-weil-2}, we deduce that there is a 
natural number $m$ depending only on $d,\epsilon$ such that $|-mK_X|$ defines a 
birational map. So we get a new proof of [\ref{B-compl}, Theorem 1.2] 
which is one of the main results of that paper. The two proofs have obvious similarities 
but also quite different in some sense. The proof in [\ref{B-compl}] relies heavily on 
boundedness of complements in dimension $d$. Indeed the birational boundedness of $|-mK_X|$ 
and the boundedness of complements are proved together in an inductive process.
The current proof relies on the BAB 
[\ref{B-BAB}, Theorem 1.1] in lower dimension which is reasonable as we want to apply induction on dimension: indeed, we are applying Propossition \ref{p-eff-bir-nefbig-weil-2} which 
needs Theorem \ref{t-eff-bir-general} in dimenion $d-1$ which (as we will see) in turn relies on 
Lemma \ref{l-peff-threshold} below in dimension $d-1$ hence we need BAB in dimension $d-1$.

(4) Now assume that $X$ is a projective $\epsilon$-lc variety of dimension $d$, for $\epsilon>0$, with $K_X$ ample 
and let $N:=2K_X$. Assuming Theorem \ref{t-eff-bir-general} in lower dimension and
applying Proposition \ref{p-eff-bir-nefbig-weil-2}, we deduce that there is a
natural number $m$ depending only on $d,\epsilon$ such that $|mK_X|$ defines a 
birational map. This is a special case of [\ref{HMX2}, Theorem 1.3]. 
The proof here is not that different because in this special setting many of the 
complications of the current proof disappear and the proof simplifies to that of [\ref{HMX2}, Theorem 1.3] 
except that towards the end of the proof where we have $m$ with $|mK_X|$ birational and 
$\vol(mK_X)$ bounded and we want to show $m$ is bounded, we use different arguments.

(5) Now assume $X$ is a projective $\epsilon$-lc variety of dimension $d$, for $\epsilon>0$, 
with $K_X\equiv 0$, that is, 
$X$ is Calabi-Yau. Assume $N$ is a nef and big $\Q$-divisor on $X$ such that 
$N=E+R$ where $E$ is integral and pseudo-effective and $R\ge 0$ with coefficients in 
$\{0\}\cup [\delta,\infty)$, for some $\delta>0$. 
Then assuming Theorem \ref{t-eff-bir-general} in lower dimension and
applying Proposition \ref{p-eff-bir-nefbig-weil-2}, we deduce that there is a
natural number $m$ depending only on $d,\epsilon,\delta$ such that $|mN|$ defines a 
birational map. Note that by (2), if $R\neq 0$, then we do not need to assume 
Theorem \ref{t-eff-bir-general} in lower dimension.
Also note that we can replace the $\epsilon$-lc condition with klt but for this we need to 
use the global ACC result [\ref{HMX2}, Theorem 1.5].

\subsection{Pseudo-effective threshold of nef and big divisors}
The following lemma proves to be useful in many places in this paper. 
Its proof in dimension $d$ relies on the BAB [\ref{B-BAB}, Theorem 1.1] in dimension $d$.

\begin{lem}\label{l-peff-threshold}
Let $d$ be a natural number and $\epsilon,\delta$ be positive real numbers. Then 
there is a natural number $l$ depending only on $d,\epsilon,\delta$ satisfying the 
following. Assume 
\begin{itemize}
\item $X$ is an $\epsilon$-lc variety of dimension $d$, 

\item $X\to Z$ is a contraction, 

\item $N$ is an $\R$-divisor on $X$ which is nef and big over $Z$, and  

\item $N=E+R$ where $E$ is integral and pseudo-effective over $Z$ and $R\ge 0$ with coefficients in 
$\{0\}\cup [\delta,\infty)$.
\end{itemize}
Then $K_X+lN$ is big over $Z$.
\end{lem}
\begin{proof}
Taking a $\Q$-factorialisation we can assume $X$ is $\Q$-factorial. 
Let $t$ be the  smallest non-negative real number 
such that $K_X+tN$ is pseudo-effective over $Z$. It is enough to show $t$ is bounded from 
above because $K_X+\lceil {(t+1)}\rceil N$ is big over $Z$. In particular, we can assume $t>0$.

Consider $(X,tN)$ as a generalised pair over $Z$ with nef part $tN$. 
Then the pair is generalised $\epsilon$-lc. Since $tN$ is big over $Z$, we can run an MMP$/Z$ on $K_X+tN$ which 
ends with a minimal model $X'$ on which $K_{X'}+tN'$ is nef and semi-ample over $Z$ [\ref{BZh}, Lemma 4.4]; 
here $K_{X'}+tN'$ is the pushdown of $K_X+tN$. 
So $K_{X'}+tN'$ defines a contraction $X'\to V'/Z$ which is non-birational 
otherwise $K_X+tN$ would be big over $Z$ which means we can decrease $t$ 
keeping $K_X+tN$ pseudo-effective over $Z$, contradicting the definition of $t$ and the assumption $t>0$.

Now since $tN'$ is big over $V'$ and $K_{X'}+tN'\equiv 0/V'$, 
$K_{X'}$ is not pseudo-effective over $V'$, hence we can run an MMP$/V'$ on $K_{X'}$ 
which ends with a Mori fibre space $X''\to W''/V'$. 
Since $(X,tN)$ is generalised $\epsilon$-lc, $(X',tN')$ is generalised $\epsilon$-lc 
which in turn implies $(X'',tN'')$ is generalised $\epsilon$-lc because $K_{X'}+tN'\equiv 0/V'$ 
(note that the nef parts of both 
$(X',tN')$ and $(X'',tN'')$ are pullbacks of $tN$ to some common resolution of $X,X',X''$). Thus  
$X''$ is $\epsilon$-lc because for each prime divisor $D$ over $X''$ we have 
$$
a(D,X'',0)\ge a(D,X'',tN'')\ge \epsilon.
$$

Let $F''$ be a general fibre of $X''\to W''$. Then $F''$ is an $\epsilon$-lc Fano variety.
Restricting to $F''$ we get 
$$
K_{F''}+tN_{F''}=(K_{X''}+tN'')|_{F''}\equiv 0
$$
where $N_{F''}=N''|_{F''}$. Moreover, $N_{F''}=E_{F''}+R_{F''}$ where 
$E_{F''}=E''|_{F''}$, $R_{F''}=R''|_{F''}$ and $E'',R''$ are the pushdowns of $E,R$. 
Then $E_{F''}$ is integral and the coefficients of 
$R_{F''}$ are in $\{0\}\cup [\delta,\infty)$. On the other hand, since $E''$ is 
pseudo-effective over $W''$, $E_{F''}$ is pseudo-effective. 
Therefore, replacing $X,N,E,R$ with $F'',N_{F''},E_{F''},R_{F''}$ and replacing $Z$ with a point, 
we can assume that $X$ is an $\epsilon$-lc Fano variety and that $K_X+tN\equiv 0$.
 
Now by [\ref{B-BAB}, Theorem 1.1], $X$ belongs to a bounded family of varieties. 
Thus there is a very ample divisor $A$ on $X$ with $-K_X\cdot A^{d-1}$ bounded from above. 
If $R\neq 0$, then $E\cdot A^{d-1}\ge 0$ and $R\cdot A^{d-1}\ge \delta$ because $A^{d-1}$ can be 
represented by a curve inside the smooth locus of $X$. But if $R=0$, then $N\cdot A^{d-1}=E\cdot A^{d-1}\ge 1$. 
Letting $\lambda=\min\{1,\delta\}$, we then have 
$$
N\cdot A^{d-1}=(E+R)\cdot A^{d-1}\ge \lambda.
$$ 
 Thus 
$$
t=\frac{-K_X\cdot A^{d-1}}{N\cdot A^{d-1}}\le -\frac{1}{\lambda}K_X\cdot A^{d-1}
$$ 
is bounded from above. Finally note that $\lambda$ depends only on $\delta$ while 
$-K_X\cdot A^{d-1}$ depends only on $d,\epsilon$.

\end{proof}

\subsection{Proofs of \ref{t-eff-bir-general}, \ref{t-eff-bir-nefbig-weil},
\ref{cor-eff-bir-nefbig-general}, \ref{cor-eff-bir-nefbig-weil}, \ref{cor-cy}}

\begin{proof}(of Theorem \ref{t-eff-bir-general})
We apply induction on dimension so assume the theorem holds in lower dimension.
Taking a small $\Q$-factorialistion we can assume that $X$ is $\Q$-factorial. 
Replacing $N$ with $2N$ we can assume $N-K_X$ is big.
By Lemma \ref{l-peff-threshold}, there is a natural number $l$ depending only on $d,\epsilon,\delta$ 
such that $K_{X}+lN$ is big. Replacing $l$ we can assume that $l\ge \frac{1}{\delta}$. 
There is an $\R$-divisor $B\le \frac{1}{2}R$ such that  
$$
M:=K_X+B+3lN
$$ 
is a $\Q$-divisor and $(X,B)$ is $\frac{\epsilon}{2}$-lc. 
In particular, $(X,B+3lN)$ is generalised $\frac{\epsilon}{2}$-lc where the nef part is $3lN$.

Now $\rddown{K_X+2lN}$ is big because if $Q$ is the fractional part of $K_X+2lN$, then 
$$
\rddown{K_X+2lN}=K_X+2lN-Q=K_X+lN+lE+lR-Q
$$
where $lR-Q\ge 0$ as $Q$ is supported in $\Supp R$ and $lR$ has non-zero coefficients $\ge 1$. 
Thus letting 
$$
F:=\rddown{K_X+2lN}+lE ~~~~\mbox{and}~~~~~ S=B+Q+lR
$$ 
we see that 
$M=F+S$ where $F$ is big and integral and the non-zero coefficients of $S$ are $\ge 1$.

We can run an MMP on $M$ ending with a minimal model $X'$ [\ref{BZh}, Lemma 4.4]. Then 
$(X',B'+3lN')$ is generalised $\frac{\epsilon}{2}$-lc which implies that $X'$ is $\frac{\epsilon}{2}$-lc.  
Moreover, under our assumptions, $M'$ is a nef and big $\Q$-divisor, and 
$M'-K_{X'}$, and $M'+K_{X'}$ are big, and 
$M'=F'+S'$ where $F'$ is big and integral and the non-zero coefficients of $S'$ are $\ge 1$. 
Therefore, applying Proposition \ref{p-eff-bir-nefbig-weil-2}, 
there is a natural number $n$ depending only on $d,\epsilon,\delta$ such that $|nM'|$ defines 
a birational map. Then $|nM|$ also defines a birational map. 

Let $m=(20dn+2)l$. Let $L$ be any pseudo-effective integral divisor. We will show that $|m'N+L|$ and $|K_X+m'N+L|$ 
define birational maps for any natural number $m'\ge m$. 
Since $B\le \frac{1}{2}R$ and $N$ is big and $E$ is pseudo-effective, 
$$
N-B=E+R-B
$$ 
is big. Moreover, if $P$ is the fractional part of 
$m'N+L$, then $P$ is supported in $R$ and $P<lR$. Now since $|nM|$ defines a birational map, 
$4dnM+L$ and $4dnM+L-K_X$ are potentially birational. Moreover, $lN-K_X$ and $lN-B$ are big 
and $3lN=M-K_X-B$. Then
$$
\rddown{m'N+L}=m'N+L-P=(20dn+1)lN+L-P+(m'-m+l)N
$$
$$
=4dn(M-K_X-B)+8dnlN+lN+L-P+(m'-m+l)N
$$
$$
=4dnM+L+4dn(lN-K_X)+4dn(lN-B)+lE+lR-P+(m'-m+l)N
$$
is potentially birational. Similar reasoning shows that $\rddown{m'N+L-K_X}$ is potentially birational
where we make use of the extra $(m'-m+l)N$ at the end of the above formula and the fact that $N-K_X$ is big.

Therefore, $|K_X+\rddown{m'N+L}|$ and $|\rddown{m'N+L}|$ define birational maps by [\ref{HMX1}, Lemma 2.3.4]. 
This in turn implies that $|K_X+{m'N+L}|$ and $|{m'N+L}|$ define birational maps. 

\end{proof}

\begin{proof}(of Theorem \ref{t-eff-bir-nefbig-weil})
This follows from Theorem \ref{t-eff-bir-general}.

\end{proof}

\begin{proof}(of Corollary \ref{cor-eff-bir-nefbig-general})
By Lemma \ref{l-peff-threshold}, there is a natural number $n>1$ depending only on $d,\epsilon,\delta$ 
such that $K_X+nN$ is big. Let $r\in\N$ be the smallest number such that $r\delta\ge 1$. 
Let $$
M:=K_X+(n+2r)N
$$
and 
$$
F:=\rddown{K_X+(n+r)N+rE},
$$ 
and $T=M-F$. Then $F$ is big because if $P$ 
is the fractional part of 
$$
K_X+(n+r)N+rE,
$$
 then 
$$
F=K_X+(n+r)N+rE-P=K_X+nN+2rE+rR-P
$$
where $rR-P\ge 0$ as $P$ is supported in $R$ and the non-zero coefficients of $rR$ are $\ge 1$. 
Moreover, 
$$
T=M-F=K_X+(n+2r)N-(K_X+(n+r)N+rE-P)=rR+P
$$ 
is supported in $R$ with any non-zero  
coefficient $\ge 1$.

Considering $(X,(n+2r)N)$ as a generalised pair with nef part $(n+2r)N$, running an MMP on $M$  
ends with a minimal model $X'$ [\ref{BZh}, Lemma 4.4]. Then $M'=F'+T'$ is nef and big 
where $F'$ is integral and big and $T'\ge 0$ with any non-zero  
coefficient $\ge 1$, $M'-K_{X'}$ is big, and $X'$ is $\epsilon$-lc. 
Applying Theorem \ref{t-eff-bir-general}, there is a 
natural number $m$ depending only on $d,\epsilon$ such that $|mM'|$ defines a birational map. 
Therefore, $|mM|$ also defines a birational map. 

Now since $|mM|$ defines a birational map, $3dmM$ is potentially birational. 
Replacing $m$ with  $3dm+1$, we can assume $mM-K_X$ is potentially birational. 
Let $l=n+3r$. Let $L$ be an integral pseudo-effective divisor, and pick natural numbers $m'\ge m$ and $l'\ge lm'$.
If $Q$ is the fractional part of  $m'K_X+l'N+L-K_X$, then we see that 
$$
\rddown{m'K_X+l'N+L-K_X}=m'K_X+l'N+L-K_X-Q
$$
$$
=m'M+(l'-lm'+rm')N+L-K_X-Q 
$$
$$
=mM-K_X+(m'-m)M+(l'-lm')N+L+rm'N-Q    
$$
is potentially birational because $rm'R-Q\ge 0$ as $Q$ is supported in $R$. 
Therefore, 
$$
|\rddown{m'K_X+l'N+L}|
$$ 
defines a birational map by [\ref{HMX2}, 2.3.4] which in turn implies that 
$$
|m'K_X+l'N+L|
$$ 
defines a birational map.

\end{proof}

\begin{proof}(of Corollary \ref{cor-eff-bir-nefbig-weil})
This is a special case of Corollary \ref{cor-eff-bir-nefbig-general}.

\end{proof}

\begin{proof}(of Corollary \ref{cor-cy})
Since $K_X+B\equiv 0$ and $N$ is big, there is a minimal model of $N$. Replacing $X$ 
with the minimal model we can assume $N$ is nef and big. Taking a $\Q$-factorialisation we 
can assume that $X$ is $\Q$-factorial. On the other hand, there is a positive real number $\epsilon$ 
depending only on $d,\Phi$ such that $X$ is $\epsilon$-lc: this follows from the global ACC 
result of [\ref{HMX2}], see for example [\ref{B-compl}, Lemma 2.48]. 
Now  $N-K_X\equiv N+B$ is big, so we can apply Theorem \ref{t-eff-bir-nefbig-weil}.
  
\end{proof}

Note that in the previous proof if $B=0$, then $N+K_X$ is big so instead of \ref{t-eff-bir-nefbig-weil} 
we can apply \ref{p-eff-bir-nefbig-weil-2} if we assume \ref{t-eff-bir-general} in lower dimension.\\


\section{\bf Birational boundedness on pseudo-effective pairs}

In this section we treat birational boundedness of divisors on pairs $(X,B)$ with pseudo-effective $K_X+B$.

\subsection{Main result}
The main result of this section is the following more general form of Theorem \ref{t-eff-bir-nefbig-klt}.

\begin{thm}\label{t-eff-bir-nefbig-klt-general}
Let $d$ be a natural number, $\delta$ be a positive real number, and $\Phi\subset [0,1]$ be a DCC set of rational numbers. 
Then there is a natural number $m$ depending only on $d,\delta,\Phi$ satisfying the following.
Assume 
\begin{itemize}
\item $(X,B)$ is a klt projective pair of dimension $d$, 

\item the coefficients of $B$ are in $\Phi$, 

\item $N$ is a nef and big $\R$-divisor, 

\item $N-(K_X+B)$ and $K_X+B$ are pseudo-effective, and 

\item $N=E+R$ where $E$ is integral and pseudo-effective  and $R\ge 0$ with coefficients in 
$\{0\}\cup [\delta,\infty)$.
\end{itemize}
Then $|m'N+L|$ and $|K_X+m'N+L|$ define birational maps for any natural number $m'\ge m$ and any 
integral pseudo-effective divisor $L$.
\end{thm}

\subsection{Pseudo-effective log divisors}
For a real number $b$ and a natural number $l$ let $b_{\rddown{l}}:=\frac{\rddown{lb}}{l}$.
Similarly for an $\R$-divisor $B$ and a natural number $l$ let $B_{\rddown{l}}:=\frac{\rddown{lB}}{l}$. 
The following statement was proved in 
[\ref{HMX2}] when $K_X+B$ is big (it follows from [\ref{HMX2}, Lemma 7.3]). 
We extend it to the case when $K_X+B$ is only pseudo-effective.

\begin{prop}\label{p-pseudo-eff-rddwn}
Let $d$ be a natural number and $\Phi\subset [0,1]$ be a DCC set of rational numbers. 
Then there is a natural number $l$ depending only on $d,\Phi$ satisfying the following. Assume 
\begin{itemize}
\item $(X,B)$ is an lc projective pair of dimension $d$,

\item the coefficients of $B$ are in $\Phi\cup (\frac{l-1}{l},1]$, and 

\item $K_X+B$ is pseudo-effective.
\end{itemize}
Then $K_X+B_{\rddown{l}}$ is pseudo-effective. 
\end{prop}
\begin{proof}
\emph{Step 1.} 
In this step we introduce some notation.
Adding $1$ to $\Phi$ we can assume $1\in \Phi$. 
Assume the proposition does not hold. Then for each $l\in \N$ there is a pair $(X^l,B^l)$ such that 
$(X^l,B^l)$ is lc projective of dimension $d$, the coefficients of $B^l$ are in $\Phi\cup (\frac{l-1}{l},1]$, 
and $K_{X^l}+B^l$ is pseudo-effective but such that $K_{X^l}+B^l_{\rddown{l}}$ is not pseudo-effective. 
Let $(V^l,R^l)$ be a $\Q$-factorial dlt model of $(X^l,B^l)$. Then $(V^l,R^l)$ is lc projective of dimension $d$, the coefficients of $R^l$ are in $\Phi\cup (\frac{l-1}{l},1]$, 
and $K_{V^l}+R^l$ is pseudo-effective. Moreover, the pushdown of $R^l_{\rddown{l}}$ is just $ B^l_{\rddown{l}}$, so  $K_{V^l}+R^l_{\rddown{l}}$ is not pseudo-effective otherwise  $K_{X^l}+B^l_{\rddown{l}}$ would be pseudo-effective.
Thus replacing $(X^l,B^l)$ with $(V^l,R^l)$ we can assume $(X^l,B^l)$ is $\Q$-factorial dlt.

Also increasing coefficients of $B^l$ in $(\frac{l-1}{l},1)$ slightly we can assume that 
$B^l$ is a $\Q$-divisor; note that this does not change $B^l_{\rddown{l}}$ because for any $b\in (\frac{l-1}{l},1)$, 
we have $b_{\rddown{l}}=\frac{l-1}{l}$.

Let $\Psi$ be the union of the coefficients of all the $B^l$. Then $\Psi$ is a DCC set: since $\Phi$  
is DCC it is enough to check that $\Psi\setminus \Phi$ is DCC; the latter follows from the fact that if  
$b_l\in \Psi\setminus \Phi$ is a coefficient of $B^l$, then  
$b_l\in (\frac{l-1}{l},1]$, so any infinite sequence of such coefficients approaches $1$, hence cannot be a 
strictly decreasing sequence.\\

\emph{Step 2.} 
In this step we show certain sets of coefficients are DCC.
Suppose that for each $l$ we have a boundary $C^l$ such that $B^l_{\rddown{l}}\le C^l\le B^l$. 
We argue that the set of 
coefficients of all the $C^l$ put together satisfies DCC. Assume not. Then there is an infinite 
subset $L\subset \N$ of numbers that for each $l\in L$ we can  
 pick a coefficient $c^l$ of $C^l$ such that the 
$c^l$ form a strictly decreasing sequence, that is, $c_{l'}<c_l$ for any $l,l'\in L$ 
with $l<l'$. For each $l$, $c^l$ is the coefficient of $C^l$ of some component, say $D^l$.
Let $b^l$ be the coefficient of $D^l$ in $B^l$.  
Replacing $L$ with an infinite subset we can assume that the 
$b^l$ with $l\in L$ form an increasing sequence approaching a limit $b$. But then the numbers $c^l$
also approach $b$ as $l$ goes to $\infty$ because 
$$
b^l-\frac{1}{l} < b^l_{\rddown{l}}\le c^l\le b^l
$$ 
 where the first inequality follows from $lb^l-1<\rddown{lb^l}$, a contradiction.\\  

\emph{Step 3.} 
In this step we run an MMP and reduce to the case in which we have a Mori fibre space structure 
$X^l\to T^l$ for all but finitely many $l$.
For ease of notation in this step we write $\Delta^l:=B^l_{\rddown{l}}$. Since $K_{X^l}+\Delta^l$ is not 
pseudo-effective and since $({X^l},\Delta^l)$ is $\Q$-factorial dlt, we can run an MMP on $K_{X^l}+\Delta^l$, 
with scaling of some ample divisor, ending with a Mori fibre space  
$Y^l\to T^l$ [\ref{BCHM}]. Denote the pushdowns of $B^l,\Delta^l$ 
to $Y^l$ by $B_{Y^l},\Delta_{Y^l}$. 

By assumption, $K_{Y^l}+B_{Y^l}$ is pseudo-effective, hence it is 
nef over $T^l$. We claim that  
$(Y^l,B_{Y^l})$ is lc for all but finitely many $l$. Assume not. Then there is an infinite subset $L\subset \N$ 
such that for each $l\in L$, $(Y^l,B_{Y^l})$ is not lc.
Then for each $l\in L$ we have a boundary $C_{Y^l}$ such that 
\begin{itemize}
\item $\Delta_{Y^l}\le C_{Y^l}\le B_{Y^l}$, 
\item $(Y^l,C_{Y^l})$ is lc, 
\item some component $D^l$ of $C_{Y^l}$ contains a non-klt centre of $(Y^l,C_{Y^l})$, and 
\item if $c^l,b^l$ are the coefficients of $D^l$ in $C_{Y^l}$ and $B_{Y^l}$ respectively, then $c^l<b^l$.
\end{itemize}
By the previous step, the set of the coefficients of all the $C_{Y^l}$ satisfies DCC.
Moreover, the set of the $c^l$ is not finite 
because otherwise replacing $L$ we can assume $c^l$ is fixed and then from 
$b^l-\frac{1}{l}<c^l<b^l$ we deduce that the $b^l$ approach $c^l$ so the set of 
the $b^l$ is not DCC, a contradiction. Therefore, replacing $L$ we can assume that the $c^l$ 
form a strictly increasing sequence.
Now $c^l$ is the lc threshold of $D^l$ with respect to the pair  
$({Y^l},C_{Y^l}-c^lD^l)$ as $D^l$ contains a non-klt centre of $(Y^l,C_{Y^l})$. 
Moreover, the set of the coefficients of all the $C_{Y^l}-c^lD^l$ satisfies DCC.
Therefore, we get a contradiction by the ACC for lc thresholds [\ref{HMX2}, Theorem 1.1].

For those $l$ such that $(Y^l,B_{Y^l})$ is lc we replace $(X^l,B^l)$ with $(Y^l,B_{Y^l})$ (then $(X^l,B^l)$ may no longer be dlt but $({X^l},B^l_{\rddown{l}})$ is still dlt). 
Therefore,  we can assume that for all but finitely many $l$ we have a Mori fibre space 
structure $X^l\to T^l$ such that $K_{X^l}+B^l_{\rddown{l}}$ is 
anti-ample over $T^l$.\\

\emph{Step 4.} 
In this step we derive a contradiction.
For each $l$ as in the previous paragraph (that is, those for which we have $X^l\to T^l$) 
we can find a boundary $\Theta^l$ such that 
$B^l_{\rddown{l}}\le \Theta^l\le B^l$ and $K_{X^l}+\Theta^l\equiv 0/T^l$. 
By Step 2, the set $\Omega$ of the coefficients of all such $\Theta^l$ form a DCC set. 
Let $F^l$ be a general fibre of $X^l\to T^l$ and let $\Theta_{F^l}=\Theta^l|_{F^l}$ (if $\dim T^l=0$, then $F^l=X^l$). 
Then $(F^l,\Theta_{F^l})$ is lc, $K_{F^l}+\Theta_{F^l}\equiv 0$, and the coefficients of 
$\Theta_{F^l}$ belong to $\Omega$ because 
if $\Theta^l=\sum a_{i}^l B_{i}^l$ where $B_{i}^l$ are the irreducible components, 
then $\Theta_{F^l}=\sum a_{i}^lB_{i}^l|_{F^l}$ where $B_{i}^l|_{F^l}$ are reduced with no common 
components for distinct $i$. 
By the global ACC [\ref{HMX2}, Theorem 1.5], the set of the coefficients of all the 
$\Theta_{F^l}$ is finite, hence the set of the horizontal (over $T^l$) coefficients of all the $\Theta^l$ is also 
finite. 

Since $K_{X^l}+B^l_{\rddown{l}}$ is 
anti-ample over $T^l$,  we can find a horizontal component of $B^l$ with coefficients 
$b^l_{\rddown{l}}, a^l,b^l$ in $B^l_{\rddown{l}},\Theta^l,B^l$, respectively, such that 
$b^l_{\rddown{l}}<a^l$.
Since the $a^l$ belong to a finite set, the $b^l$ also belong to a finite set otherwise they would not form a 
DCC set. Thus there is a natural number $p$ such that $pb^l$ is integral for all such $b^l$. 
But then for every $l$ divisible by $p$, we have 
$$
b^l_{\rddown{l}}=\frac{\rddown{lb^l}}{l}=\frac{{lb^l}}{l}=b^l,
$$ 
contradicting $b^l_{\rddown{l}}<a^l\le b^l$.
 
\end{proof}

\subsection{Proofs of \ref{t-eff-bir-nefbig-klt-general} and \ref{t-eff-bir-nefbig-klt}}

\begin{proof}(of Theorem \ref{t-eff-bir-nefbig-klt-general})
\emph{Step 1.}
In this step we introduce some notation. 
Let $l$ be as in Proposition \ref{p-pseudo-eff-rddwn} for the data $d,\Phi$. Let $(X,B),N,E,R$ be as in 
Theorem \ref{t-eff-bir-nefbig-klt-general}.
Replacing $X$ with a $\Q$-factorialisation we 
can assume that $X$ is $\Q$-factorial. Let $X'\to X$ be the birational map which extracts exactly
the exceptional prime divisors $D$ over $X$  with log discrepancy 
$$
a(D,X,B) <\frac{1}{2l}.
$$  
Here by exceptional prime divisors over $X$ we mean prime divisors on birational models of 
$X$ which are exceptional over $X$.
If there is no such $D$, then $X'\to X$ is the identity morphism.
Let $K_{X'}+B',N',E',R'$ be the pullbacks of $K_X+B, N, E,R$ to $X'$, respectively.\\ 

\emph{Step 2.}
In this step we study $\Delta':=B'_{\rddown{l}}$. By construction, $K_{X'}+B'$ is pseudo-effective 
and the coefficients of $B'$ belong to $\Phi\cup (1-\frac{1}{2l},1)$ because the non-exceptional$/X$ 
components of $B'$ have coefficients in $\Phi$ and the exceptional$/X$ components 
$D$ of $B'$ have coefficients  
$$
\mu_DB'=1-a(D,X,B)\in (1-\frac{1}{2l},1)
$$ 
as 
$$
a(D,X,B)\in (0,\frac{1}{2l}).
$$
Thus by Proposition \ref{p-pseudo-eff-rddwn}, $K_{X'}+\Delta'$ is pseudo-effective.
Moreover, for exceptional $D$, 
$
\mu_D\Delta'=1-\frac{1}{l} 
$ 
by definition of $\Delta'$, hence $\mu_D(B'-\Delta')> \frac{1}{2l}$.

On the other hand, by our choice of $X'\to X$, for any exceptional prime divisor $C$ over 
$X'$,  we have  
$$
a(C,X',B')=a(C,X,B)\ge \frac{1}{2l}.
$$
Thus $(X',0)$ is $\frac{1}{2l}$-lc.\\

\emph{Step 3.}
In this step we introduce some notation. Let $r\in \N$ be the smallest number such that $r\delta\ge 1$. 
Let  
\begin{itemize}

\item $S'$ be the sum of all the exceptional$/X$ prime divisors on $X'$,

\item  $J'=\Supp (S'+R')$,

\item $F'=\rddown{6lN'-2S'+J'}$, 

\item $P'$ be the fractional part of $6lN'-2S'+J'$, 

\item $T'=6lN'-F'$, 

\item $G'=\rddown{ 2rE'}$, 

\item $Q'$ be the fractional part of $2rE'$, and  

\item $V'=2rR'+Q'$, 

\end{itemize}

Note that $P'$ is supported in $J'$ because $P'$ is also the fractional part of $6lN'$ 
and because any component of $6lN'=6lE'+6lR'$ with non-integral 
coefficient is either a component of $R'$ or an exceptional component of $E'$ as $E$ is integral. 
Moreover, $Q'$ is exceptional over $X$ because $2rE$ is integral, hence $Q'$ is supported in $S'$.\\

\emph{Step 4.}
In this step we show that $F'+G'$ is pseudo-effective and that $T'+V'$ is supported in $J'$ whose non-zero coefficients $\ge 1$.
Let 
$$
A'=N'-(K_{X'}+B')
$$ 
which is pseudo-effective as $N-(K_X+B)$ is pseudo-effective by assumption.
We then have 
$$
N'=K_{X'}+B'+A'=K_{X'}+\Delta'+A'+B'-\Delta'
$$
where $K_{X'}+\Delta'+A'$ is pseudo-effective as $K_{X'}+\Delta'$ is pseudo-effective by Step 2, 
and $B'-\Delta'\ge 0$ with coefficients of exceptional 
components $> \frac{1}{2l}$ again by Step 2. In particular, 
$$
2l(B'-\Delta')-S'\ge 0
$$ 
and so $2lN'-S'$ is pseudo-effective.

Now
$$
F'+G'=\rddown{6lN'-2S'+J'}+\rddown{2rE'}
$$
$$
=6lN'-2S'+J'-P'+2rE'-Q'
$$
$$
=6lN'-3S'+2rE'+J'-P'+S'-Q',
$$
is pseudo-effective because $6lN'-3S'$ is pseudo-effective by the previous paragraph, $2rE'$ is pseudo-effective, and 
$$
J'-P'+S'-Q'\ge 0
$$ 
as $J'$ contains the support of $P'$ and as $S'$ contains the support of $Q'$ 
by Step 3. 

On the other hand,
$$
T'+V'=6lN'-F'+2rR'+Q'
$$
$$
=6lN'-(6lN'-2S'+J'-P')+2rR'+Q'
$$
$$
=2S'-J'+P'+2rR'+Q'
$$
$$
\ge 2S'-J'+2rR'
$$
where the  inequality follows from $P',Q'\ge 0$. In particular, $T'+V'$ is supported in $J'$. 
Moreover,  $S'-J'+rR'\ge 0$ because the non-zero 
coefficients of $S'+rR'$ are $\ge 1$: this is clear for the exceptional components; 
and for the non-exceptional components it follows from $r\delta \ge 1$ and the assumption 
that the non-zero coefficients of $R$ are $\ge \delta$. Therefore, 
$$
T'+V'\ge 2S'-J'+2rR'\ge S'-J'+rR'+S'+rR'\ge S'+rR'\ge J',
$$
so $\Supp(T'+V')=J'$, and the non-zero coefficients of $T'+V'$ are $\ge 1$ by the previous sentence.\\ 

\emph{Step 5.}
In this step we finish the proof by applying \ref{t-eff-bir-general}.
From the equalities in Step 4 we see that 
$$
F'+G'+T'+V' 
$$
$$
=(6lN'-2S'+J'-P'+2rE'-Q')+(2S'-J'+P'+2rR'+Q')
$$
$$
=6lN'+2rE'+2rR'
$$
$$
=(6l+2r)N'.
$$ 
Also recall from Step 2 that $X'$ is $\frac{1}{2l}$-lc. Moreover, $(6l+2r)N'-K_{X'}$ is big 
because $N'-K_{X'}=B'+A'$ is pseudo-effective by the first paragraph of Step 4. 

Therefore, applying Theorem \ref{t-eff-bir-general} to 
$$
{ X', ~~~~~(6l+2r)N'=F'+G'+T'+V'},
$$ 
we deduce that there is a natural number $n$ 
depending only on $d,\frac{1}{2l}$ such that the linear system 
$
|n(6l+2r)N'|
$ 
defines a birational map. In particular, $3dn(6l+2r)N'$ is potentially birational which 
in turn implies that $3dn(6l+2r)N$ is potentially birational.

We will show that 
$$
m:=3dn(6l+2r)+r+2
$$ 
satisfies the theorem.
Let $L$ be any pseudo-effective integral divisor on $X$.
Pick a natural number $m'\ge m$. Let $I$ be the fractional part of $m'N$. Then 
$$
\rddown{m'N+L}=m'N-I+L
$$
$$
=(3dn(6l+2r)+r+2)N+(m'-m)N-I+L
$$
$$
=3dn(6l+2r)N+rE+(m'-m+2)N+L+rR-I
$$
is potentially birational because 
$$
rE+(m'-m+2)N+L
$$ 
is big and $rR-I\ge 0$ as $I$ is supported in $R$, 
 $r\delta \ge 1$, and the non-zero coefficients of $R$ are $\ge \delta$.
Similar reasoning shows that $\rddown{m'N+L-K_{X}}$ is potentially birational where we make use of 
the fact that $N-K_{X}$ is pseudo-effective. Therefore, by [\ref{HMX1}, Lemma 2.3.4],
$$
|\rddown{m'N+L}| ~~~~\mbox{and}~~~~ |K_X+\rddown{m'N+L}|
$$ 
define briational maps which in turn imply that 
$$
|m'N+L| ~~~~\mbox{and}~~~~ |K_{X}+m'N+L|
$$ 
define briational maps. 
Finally note that $m$ depends only on $d,\Phi,\delta$ because 
$n$ depends only on $d,\frac{1}{2l}$ and because $l,r$ depend only 
on $d,\Phi, \delta$.

\end{proof}

\begin{proof}(of Theorem \ref{t-eff-bir-nefbig-klt})
This follows from Theorem \ref{t-eff-bir-nefbig-klt-general}.

\end{proof}


\section{\bf Boundedness of polarised pairs}

In this section we treat boundedness of polarised pairs, namely we prove \ref{t-bnd-nef-pol-pairs}, 
\ref{cor-bnd-cy}, \ref{t-lct-nef-big-slc-cy}, 
and \ref{cor-bnd-semi-lc-pol-CY}. Bounding certain lc thresholds plays a key role in the proofs of all these results.
This bounding is achieved through a combination of birational boundedness of linear systems, birational boundedness of 
pairs, and boundedness of lc thresholds on bounded families.

\subsection{Polarised nef $\epsilon$-lc pairs}
We begin with proving a more general version of Theorem \ref{t-bnd-nef-pol-pairs}.

\begin{thm}\label{t-bnd-nef-pol-pairs-general}
Let $d$ be a natural number and $v,\epsilon,\delta$ be positive real numbers. 
Consider pairs $(X,B)$ and divisors $N$ on $X$ such that  
\begin{itemize}
\item $(X,B)$ is projective $\epsilon$-lc of dimension $d$,

\item the coefficients of $B$ are in $\{0\}\cup [\delta,\infty)$,

\item $K_X+B$ is nef, 

\item $N$ is a nef and big $\R$-divisor, 

\item $N=E+R$ where $E$ is integral and pseudo-effective  and $R\ge 0$ with coefficients in 
$\{0\}\cup [\delta,\infty)$, and 

\item $\vol(K_X+B+N)\le v$.
\end{itemize}
Then the set of such $(X,\Supp B)$ forms a bounded family. 
If in addition $N\ge 0$, then the set of such $(X,\Supp(B+N))$ forms a bounded family.
\end{thm}

\begin{proof}
\emph{Step 1.} In this step we will define a divisor $M$ and study some of its properties.
Applying Corollary \ref{cor-eff-bir-nefbig-general} on a small $\Q$-factorialisation of $X$, there exist natural numbers 
$m,l\ge \frac{1}{\delta}$ depending only on $d,\epsilon,\delta$ such that the linear system 
$$
|mK_{X}+lmN+mE|
$$  
defines a birational map. Pick an element 
$L$ 
of this linear system and then define 
$$
M:=mB+mR+L
$$ 
which is effective.

 From here to the end of Step 4 we assume that $K_X+B+N$ is ample. 
Then 
$$
M\sim mB+mR+mK_{X}+lmN+mE=m(K_X+B+N)+lmN
$$
is ample, and $|M|$ defines a birational map as $M\ge L$. Moreover, for any component $D$ of $M$, 
we have 
$$
\mu_D(B+M)\ge \mu_D M\ge 1;
$$ 
indeed, if $D$ is not a component of the fractional part of $M$, then 
obviously $\mu_DM\ge 1$; if $D$ is a component of the fractional part of $M$, then $D$ is a component of 
$B+R$ and we have 
$$
\mu_D M\ge \mu_D(mB+mR)\ge 1
$$ 
as $m\delta\ge 1$ and as the non-zero coefficients of $B+R$ are $\ge \delta$.
In addition, 
$$
\vol(M)=\vol(m(K_X+B+N)+lmN)
$$
$$
\le \vol(m(K_X+B+N)+lm(K_X+B)+lmN)
$$
$$
= \vol((l+1)m(K_X+B+N))\le ((l+1)m)^{d}v.
$$
Also it is clear from the definition of $M$ that $M-(K_X+B)$ is big.\\

\emph{Step 2.}
In this step we show that $(X,\Supp (B+M))$ is birationally bounded. 
By Lemma \ref{l-perturbation-R-div}, there is a $\Q$-Cartier 
$\Q$-divisor $A$ such that $\Supp A=\Supp M$ and $M-A$ has arbitrarily small coefficients.  
In particular, $A\ge 0$ as $M\ge 0$, and we can assume that $A$ is ample as $M$ is ample, and that 
$$
\vol(A)\le ((l+1)m)^{d}v+1.
$$  
Moreover, we can assume that  
$2A\ge M$ which in particular means that $|2A|$ defines a birational map and 
the coefficients of $2A$ are $\ge 1$ by Step 1. In addition, $2A-(K_X+B)$ is big as $M-(K_X+B)$ is big. 
Also note that for any component $D$ of $2A$ we have 
$$
\mu_D(B+2A)\ge \mu_D(B+M)\ge 1.
$$

Now applying [\ref{B-compl}, Proposition 4.4] to $(X,B),2A$, we deduce that 
there exist a bounded set of couples $\mathcal{P}$ and a natural number $c$ 
depending only on $d,v,l,m,\delta$ such that 
we can find a projective log smooth couple
 $({\overline{X}},\overline{{\Sigma}})\in\mathcal{P}$ and a birational map $\overline{X}\bir X$ such that 
\begin{itemize}
\item  $\Supp \overline{{\Sigma}}$ contains the exceptional 
divisors of  $\overline{X}\bir X$ and the birational transform of $\Supp (B+2A)=\Supp (B+M)$;

\item if $\phi\colon X'\to X$ and $\psi\colon X'\to \overline{X}$ is a common resolution, 
then each coefficient of $\overline{A}:=\psi_*\phi^*A$ is at most $c$.\\
\end{itemize}

\emph{Step 3.}
Next we show that the lc threshold $t$ of $M$ with respect to $(X,B)$ is bounded 
from below away from zero. Let  
$$
K_{\overline{X}}+{\overline{B}}=\psi_*\phi^*(K_X+B).
$$ 
By definition of $t$, $(X,B+tM)$ is not klt, hence $(X,B+2tA)$ is also not klt as $2A\ge M$.
Then, since $K_X+B+2tA$ is ample, by the negativity lemma 
$$
\phi^*(K_X+B+2tA)\le \psi^*(K_{\overline{X}}+{\overline{B}}+2t\overline{A}),
$$
which implies that $({\overline{X}},\overline{{B}}+2t\overline{{A}})$ is not sub-klt. 
Thus 
$$
({\overline{X}},(1-\epsilon)\overline{{\Sigma}}+2t\overline{{A}})
$$ 
is not sub-klt 
as $\overline{{B}}\le (1-\epsilon)\overline{{\Sigma}}$ because $(X,B)$ is $\epsilon$-lc. 
Then since the above pair is log smooth, $t$ is bounded from below away from zero depending only on 
$\epsilon,c$. Then $t$ depends only on $d,v,l,m,\epsilon,\delta$, so we can assume 
it depends only on $d,v,\epsilon,\delta$.
Also note that $t\le 1$ because as noted above each component $D$ of $M$ 
satisfies $\mu_D M\ge 1$.\\

\emph{Step 4.}
In this step we show that $(X,\Supp (B+M))$ is bounded.
By the previous step, $(X,B+\frac{t}{2}M)$ is $\frac{\epsilon}{2}$-lc as $(X,B)$ is $\epsilon$-lc. Moreover, 
$K_X+B+\frac{t}{2}M$ is ample and the non-zero coefficients of $B+\frac{t}{2}M$ are $\ge \min\{\delta,\frac{t}{2}\}$. 
And by Step 2, $(X,\Supp (B+M))$ is birationally bounded.
Therefore, applying [\ref{HMX2}, Theorem 1.6], we deduce that $(X,\Supp (B+M))$ is bounded.

Now assume $N\ge 0$. By adding a multiple of $N$ to $M$ in Step 1, say by replacing $l$ with $l+1$, 
we can assume $M\ge N$. Thus boundedness of  
$(X,\Supp (B+M))$ implies boundedness of $(X,\Supp (B+N))$.
\\

\emph{Step 5.}
Now we treat the general case when $K_X+B+N$ is only nef and big.
Since $(X,B)$ is klt, $K_X+B$ is nef, and $N$ is nef and big, we see that $K_X+B+N$ is nef and big and semi-ample 
by the base point freeness theorem, 
hence it defines a birational contraction $X\to Y$. Then $K_X+B\equiv 0/Y$ and $N\equiv 0/Y$ as 
both $K_X+B,N$ are nef, hence by again applying 
the base point freeness theorem to $K_X+B$ over $Y$ we have $K_X+B\sim_\R 0/Y$ and $N\sim_\R 0/Y$ 
which in turn implies $M\sim_\R 0/Y$. 
By the above steps, $(Y,\Supp (B_Y+M_Y))$ is bounded where $B_Y,M_Y$ are pushdowns of $B,M$.

By construction, 
$
(X,B+\frac{t}{2}M)
$
is a crepant model of $(Y,B_Y+\frac{t}{2}M_Y)$. In the terminology of [\ref{B-lcyf}],
$$
(X,B+\frac{t}{2}M)\to Y
$$ 
is a $(d,r,\epsilon)$-Fano type log Calabi-Yau fibration for some fixed $r\in \N$. 
Therefore, applying [\ref{B-lcyf}, Theorem 1.3] we deduce that $(X,\Supp (B+M))$ is bounded.
 In particular, $(X,\Supp (B+N))$ is bounded if $N\ge 0$ 
as we can assume $M\ge N$ as in Step 4. 

\end{proof}

It is worth pointing out that we used [\ref{B-lcyf}, Theorem 1.3] in the proof but only when $K_X+B+N$ 
is not ample.

\begin{proof}(of Theorem \ref{t-bnd-nef-pol-pairs})
This is a special case of Theorem \ref{t-bnd-nef-pol-pairs-general}.

\end{proof}

\begin{proof}(of Corollary \ref{cor-bnd-cy})
Since $(X,B)$ is klt and $K_X+B\equiv 0$, by [\ref{B-compl}, Lemma 2.48], 
there is a positive real number $\epsilon$ depending only on 
$d,\Phi$ such that $(X,B)$ has $\epsilon$-lc 
singularities (note that the lemma requires $(X,0)$ to be klt but we can achieve this on a small 
$\Q$-factorialisation of $X$). Moreover, 
$$
\vol(K_X+B+N)=\vol(N)\le v.
$$ 
Now apply Theorem \ref{t-bnd-nef-pol-pairs}.

\end{proof}

\subsection{Lc thresholds on slc Calabi-Yau pairs}
Next we prove a more general version of Theorem \ref{t-lct-nef-big-slc-cy}.

\begin{thm}\label{t-lct-nef-big-slc-cy-general}
Let $d$ be a natural number, $v,\delta$ be positive real numbers, and $\Phi\subset [0,1]$ be a DCC set of real numbers. 
Then there is a positive real number $t$ depending only on $d,v,\delta,\Phi$ satisfying the following.
Assume that   
\begin{itemize}
\item $(X,B)$ is an slc Calabi-Yau pair of dimension $d$, 

\item the coefficients of $B$ are in $\Phi$,

\item $N\ge 0$ is a nef $\R$-divisor on $X$ with coefficients $\ge \delta$,

\item $(X,B+uN)$ is slc for some real number $u>0$,

\item for each irreducible component $S$ of $X$, $N|_S$ is big and $\vol(N|_S)\le v$.
\end{itemize}
Then $(X,B+tN)$ is slc. 
\end{thm}
\begin{proof}
\emph{Step 1.}
In this step we reduce the theorem to the case when $X$ is normal and irreducible. 
Let $X^{\nu}\to X$ be the normalisation of $X$ and let $K_{X^{\nu}}+B^{\nu}$ and $N^{\nu}$ be the 
pullbacks of $K_X+B$ and $N$. Since $K_X+B\sim_\R 0$, we get $K_{X^{\nu}}+B^{\nu} \sim_\R 0$. 
Recall from \ref{ss-slc-pairs} that 
$B^\nu$ is the sum of the birational transform of $B$ and the reduced conductor divisor of $X^{\nu}\to X$.
So the coefficients of $B^\nu$ belong to $\Phi\cup \{1\}$. Replacing $\Phi$ with $\Phi\cup \{1\}$ we can 
assume these coefficients are in $\Phi$. On the other hand, 
since $(X,B+uN)$ is slc for some $u>0$, $\Supp N$ does not contain any 
singular codimension one point of $X$. Thus 
$N^\nu$ is the birational transform of $N$ and the coefficients of $N^\nu$ are $\ge \delta$. 

By definition of slc pairs, for any real number $t\ge 0$, $(X,B+tN)$ is slc iff 
$({X^{\nu}},B^{\nu}+tN^{\nu})$ is lc on each irreducible component of $X^\nu$. 
By assumption, $N^{\nu}$ is nef and big on each irreducible 
component of $X^{\nu}$ with volume at most $v$. Therefore, replacing $(X,B),N$ with the 
restriction of $({X^{\nu}},B^{\nu}),N^{\nu}$ to an arbitrary irreducible component of 
$X^\nu$ we can assume $X$ is normal and irreducible. 
In particular, $(X,B+uN)$ is an lc pair for some $u>0$.\\

\emph{Step 2.}
In this step we reduce the theorem to the case when $X$ has $\epsilon$-lc singularities for some 
fixed $\epsilon>0$ depending only on $d,\Phi$.
Let $(X',B')$ be a $\Q$-factorial dlt model of $(X,B)$ 
and let $N'$ be the pullback of $N$. By definition of dlt models, each exceptional$/X$ 
prime divisor on $X'$ appears in $B'$ with coefficient $1$.   
By assumption, $(X,B+uN)$ is lc for some $u>0$,
hence $(X',B'+uN')$ is lc, so $\Supp N'$ cannot contain any exceptional divisor which means   
$N'$ is just the birational transform of $N$ so its coefficients are $\ge \delta$. 
Replacing $(X,B),N$ with $(X',B'),N'$ we can assume that $(X,0)$ is $\Q$-factorial klt. 

On the other hand, since $K_{X}+B\sim_\R 0$ and since the coefficients of 
$B$ are in the DCC set $\Phi$, there is a positive real number $\epsilon$ depending only on $d,\Phi$ such that 
if $D$ is a prime divisor over $X$ with log discrepancy $a(D,X,B)<\epsilon$, then $a(D,X,B)=0$ 
[\ref{B-compl}, Lemma 2.48]. In particular, if $D$ is a prime divisor over $X$ with 
$a(D,X,0)<\epsilon$, then $a(D,X,B)=0$.

Assume $X''\to X$ extracts exactly all the prime divisors $D$ over $X$ with $a(D,X,0)<\epsilon$. 
If there is no such $D$, then $X''\to X$ is the identity morphism.
Let $K_{X''}+B''$ be the pullback of $K_X+B$. Then each exceptional prime divisor of $X''\to X$ 
has coefficient $1$ in $B''$ by the previous paragraph. In particular, the  coefficients of $B''$ are in $\Phi$, and 
$\Supp N$ does not contain the image of 
such divisors on $X$. Thus if $N''$ is the pullback of $N$, then $N''$ is just 
the birational transform of $N$, so its coefficients are $\ge \delta$.
In addition, by construction, $(X'',0)$ has $\epsilon$-lc singularities because for any prime divisor $C$ 
on birational models of $X''$ and exceptional over $X''$, we have 
$$
a(C,X'',0)\ge a(C,X,0)\ge \epsilon
$$ 
by our choice of $X''$.
Replacing $(X,B),N$ with $(X'',B''),N''$ we can then assume that $(X,0)$ has $\epsilon$-lc singularities.\\

\emph{Step 4.}
In this step we show that $(X,\Supp (B+N))$ is birationally bounded. 
We want to apply  [\ref{B-compl}, Proposition 4.4] but since this is stated only for 
nef $\Q$-divisors, we need to apply it indirectly as follows. 
Since $X$ is $\epsilon$-lc and $N$ is nef and big with coefficients $\ge \delta$, 
by Corollary \ref{cor-eff-bir-nefbig-general}, there exist natural numbers $m,l$ 
depending only on $d,\epsilon,\delta$ such that $|mK_X+lN|$ defines a birational map. 
 Pick an element $L$ of this linear system. 
Replacing $l$ we can assume that $L\ge N+\Supp N$.   
Let $M:=m\Delta+L$ where $0\le \Delta\le N$ is a small $\R$-divisor so that 
$M$ is a $\Q$-divisor and $(X,\Delta)$ is $\frac{\epsilon}{2}$-lc.

Considering $(X,\Delta+\frac{l}{m}N)$ as a generalised pair with nef part $\frac{l}{m}N$, 
running an MMP on 
$$
K_X+\Delta+\frac{l}{m}N\sim_\Q \frac{1}{m}M
$$ 
ends with a minimal model, say $Y$.
 Thus $M_Y$ is a nef and big $\Q$-divisor. Moreover, 
 $$
 M_Y-(K_Y+B_Y)\equiv M_Y
 $$ is big.  
In addition, 
$$
\vol(M_Y)=\vol(M)=\vol(mK_X+m\Delta+lN)=\vol(-mB+m\Delta+lN)
$$
$$
\le \vol((m+l)N)\le (m+l)^dv.
$$
On the other hand, for any component $D$ of $M_Y$, 
$$
\mu_D(B_Y+M_Y)\ge \mu_DM_Y\ge 1;
$$ 
indeed, if $D$ is a not a component of the fractional part of $M_Y$, this is obvious; otherwise 
$D$ is a component of $N_Y$ in which case the assertion follows from $M_Y\ge L_Y\ge  \Supp N_Y$.

Therefore, applying [\ref{B-compl}, Proposition 4.4] to $(Y,B_Y),M_Y$,  
there is a positive real number $c$ and a bounded set of couples $\mathcal{P}$ depending 
only on $d,v,m,l,\Phi$ such that there is a projective log smooth couple $(\overline{X},{\overline{\Sigma}})\in \mathcal{P}$ 
and a birational map $\overline{X}\bir Y$ such that 
\begin{itemize}
\item  $\Supp {\overline{\Sigma}}$ contains the exceptional 
divisor of  $\overline{X}\bir Y$ and the birational transform of $\Supp (B_Y+{M_Y})$;

\item if $\rho\colon X'\to Y$ and $\psi\colon X'\to \overline{X}$ is a common resolution and 
$\overline{M}=\psi_*\rho^*M_Y$, then each coefficient of $\overline{M}$ 
is at most $c$.

\end{itemize}

Now ${\overline{\Sigma}}$ contains the exceptional divisors of the induced map ${\overline{X}}\bir X$ 
and the birational transform of $\Supp (B+{N})$ because any component of the latter is either 
exceptional over $Y$ or is the birational transform of some component of $B_Y+M_Y$. 
Moreover, replacing $X'$ we can assume 
$\phi\colon X'\to X$ is also a resolution. But then since $N$ is nef, 
$$
\overline{N}:=\psi_*\phi^*N\le \psi_*\rho^*N_Y\le \psi_*\rho^*M_Y=\overline{M},
$$
hence $\overline{N}$ is supported in $\overline{\Sigma}$ with coefficients $\le c$.\\

\emph{Step 5.}
In this step we finish the proof. 
Let 
$$
K_{X'}+B'=\phi^*(K_X+B), ~~~ N'=\phi^*N,
$$ 
and let 
$$
K_{\overline{X}}+{\overline{B}}=\psi_*\phi^*(K_X+B).
$$ 
Then $(X',B')$ is sub-lc and $(\overline{X},{\overline{B}})$ is also sub-lc as 
$\Supp \overline{B}\subset \overline{\Sigma}$.  
Since $\Supp N$ does not contain any non-klt centre of $(X,B)$, $\Supp N'$ does not contain any non-klt centre 
of $(X',B')$, hence no component of $N'$ has coefficient $1$ in $B'$. Thus 
 no component of $\overline{N}$ has coefficient $1$ in $\overline{B}$. 
 
On the other hand, by Step 2, no component of $\overline{B}$ has coefficient in 
$(1-\epsilon,1)$ otherwise we would find a prime divisor $D$ over $X$ with 
$$
0<a(D,X,B)<\epsilon
$$ 
which is not possible by our choice of $\epsilon$. 
Thus every component of $\overline{N}$ has coefficient $\le 1-\epsilon$ in $\overline{B}$.  
Also by the previous step 
the coefficients of $\overline{N}$ are at most $c$. 
Therefore, letting $t=\frac{\epsilon}{c}$ we see that 
the coefficients of ${\overline{B}}+t\overline{N}$ do not exceed $1$ because for any prime divisor $D$ 
either $\mu_D\overline{B}=1$ and $\mu_Dt\overline{N}=0$, or $\mu_D\overline{B}\le 1-\epsilon$ 
and $\mu_Dt\overline{N}\le \epsilon$. 
Moreover, since $\Supp {\overline{\Sigma}}$ contains $\Supp {\overline{B}}\cup \Supp \overline{N}$, 
$$
(\overline{X},\Supp {\overline{B}}\cup \Supp \overline{N}))
$$ 
is log smooth. 
Therefore, $(\overline{X},{\overline{B}}+t\overline{N})$ is sub-lc. 

Now since $K_X+B+tN$ is nef, by the negativity lemma, we have 
$$
\phi^*(K_X+B+tN)\le \psi^*(K_{\overline{X}}+{\overline{B}}+t\overline{N}),
$$
hence we deduce that $(X,B+tN)$ is lc. 
Note that we can assume that $t$ 
depends only on $d,v,\delta,\Phi$ because $\epsilon$ depends only on $d,\Phi$, 
and $m,l$ depend only on $d,\epsilon,\delta$, and $c$ depends only on 
$d,v,m,l,\Phi$.

\end{proof}

\begin{proof}(of Theorem \ref{t-lct-nef-big-slc-cy})
This is a special case of Theorem \ref{t-lct-nef-big-slc-cy-general}.

\end{proof}

\subsection{Polarised slc Calabi-Yau pairs}

\begin{proof}(of Corollary \ref{cor-bnd-semi-lc-pol-CY})
If $X_i$ are the irreducible components of $X$, then $\vol(N)=\sum \vol(N|_{X_i})$, hence 
$\vol(N|_{X_i})\le v$ for each $i$. Thus 
by Theorem \ref{t-lct-nef-big-slc-cy}, there is a rational number $t>0$ depending only on $d,v,\Phi$ such that 
$(X,B+tN)$ is slc. Since the coefficients of $B$ are in the DCC set $\Phi$ and since $N\ge 0$ is integral, 
the coefficients of $B+tN$ belong to a DCC set depending only on $\Phi,t$. Moreover, $K_X+B+tN$ is ample with  
$$
\vol(K_X+B+tN)=\vol(tN)=t^dv.
$$ 
Therefore, we can apply [\ref{HMX3}, Theorem 1.1] to deduce that $(X,\Supp (B+tN))$ belongs to a bounded family.

\end{proof}


\section{\bf Further remarks}\label{s-remarks}

In this final section we present some examples and remarks related to some of the results in this paper.

\begin{exa}\label{exa-2}
\emph{ 
This example shows that we cannot drop the condition $N-K_X$ being pseudo-effective 
in Theorem \ref{t-eff-bir-nefbig-weil} in general. 
Assume $X$ is a smooth projective curve and $N$ is one point on $X$. 
Let $m$  be the smallest natural number such that $|mN|$ defines a birational map. 
In general $m$ is not bounded. Indeed, assume not, that is, assume $m$ is 
bounded from above. Then $\vol(mN)=m$ is bounded, hence $X$ is 
birationally bounded [\ref{HMX1}, Lemma 2.4.2(2)] 
which implies that $X$ is bounded, a contradiction as $X$ is an arbitrary smooth projective curve. 
}
\end{exa}

\begin{exa}\label{exa-3}
\emph{
This example shows that the condition   
$X$ having $\epsilon$-lc singularities in Theorem \ref{t-eff-bir-nefbig-weil} cannot be replaced with just 
assuming $X$ having klt singularities.
Let $X$ be the weighted projective surface $\PP(p,q,r)$ where $p,q,r$ are coprime natural numbers. 
Then $X$ is a toric Fano surface with klt singularities. Let $N=-K_X$. Then  
$$
\vol(N)=\frac{(p+q+r)^2}{pqr}
$$
can get arbitrarily small meaning that there is no positive lower bound on $\vol(N)$ [\ref{HMX2}, Example 2.1.1]. 
Thus if $m$ is a natural number such that $|mN|$ defines a birational map, then 
there is no upper bound on $m$ because $\vol(N)\ge \frac{1}{m^2}$.
}
\end{exa}

\begin{exa}\label{exa-3'}
\emph{
This example shows that we cannot drop the nefness of $N$ in Theorem \ref{t-eff-bir-nefbig-weil}.
Let $X$ be as in Example \ref{exa-3} and $N=-K_X$. Take the minimal resolution $\phi\colon W\to X$ and write 
$K_W+E=\phi^*K_X$. 
Let $N_W=\lceil \phi^*N\rceil$. Then $N_W=\phi^*N+G$ for some $G\ge 0$. 
By construction $W$ is smooth, $N_W$ is integral and big, 
and 
$$
N_W-K_W=\phi^*N+G -K_W
$$
$$
=  \phi^*N+G -\phi^*K_X+E=2\phi^*N+G+E
$$
is big. If $|mN_W|$ defines a birational map, then $|mN|$ also defines a birational map. 
But it was noted in \ref{exa-3} that $m$ depends on $X$ and  in general not bounded.
}
\end{exa}

\begin{rem}\label{rem-3''}
\emph{If in addition to the assumptions of Corollary \ref{cor-eff-bir-nefbig-weil} we assume that 
the Cartier index of $N$ is bounded by some fixed number $p$, then the corollary essentially follows from 
the results of [\ref{BZh}]. Indeed in this case we can find a bounded natural number $l$ 
such that $K_X+lN$ is big, so we can apply [\ref{BZh}, Theorem 1.3] to deduce that 
$|m(K_X+lN)|$ defines a birational map for some bounded natural number $m$. In this case we do not need 
the $\epsilon$-lc condition on $X$.
}
\end{rem}

\begin{exa}\label{exa-4}
\emph{
 It is tempting to try to generalise Theorem \ref{t-eff-bir-nefbig-weil} to the case when $N$ is only 
nef but not necessarily big. More precisely, assume $X$ is a projective variety with $\epsilon$-lc 
singularity with $\epsilon>0$, of dimension $d$, and $N$ is a nef integral divisor with $N-K_X$ big. Then 
one may ask whether there is $m$ depending only on $\epsilon,d$ such that $|mN|$ defines the Iitaka fibration 
associated to $N$. Such an $m$ does not exist as the following example in dimension two shows. 
Assume $Y$ is the projective cone over an elliptic curve 
$T$ defined over $\C$, and $X$ is obtained by blowing up the vertex. Then $X$ is smooth and $Y$ has an  
lc (non-klt) singularity at the vertex but smooth elsewhere. 
Then taking $a>0$ small, $-(K_X+(1+a)T)$ is positive on both extremal rays of $X$, so it is ample. 
Thus $-K_X$ is big.} 

\emph{
Pick a natural number $n$ and pick a 
torsion Cartier divisor $G$ on $T$ such that $nG\sim 0$ but $n'G\not\sim 0$ for any natural number $n'<n$ 
(we find such a $G$ by viewing $T$ as $\C/\Lambda$ for some lattice $\Lambda$). 
Let $N$ be the pullback of $G$ via $X\to T$. 
Then $|n'N|$ is empty for every natural number $n'<n$ although $N-K_X$ is big. 
Since $n$ can be arbitrarily large, there is no bounded $m$ independent of the choice of $N$ 
such that $|mN|$ defines the Iitaka fibration of $N$. 
}
\end{exa}

If we restrict ourselves to Fano type varieties, then at least conjecturally we expect much better 
behaviour with respect to the above problem. 

\begin{conj}\label{conj-nef-integral-FT}
Let $d$ be a natural number and $\epsilon$ be a positive real number. Then 
there is a natural number $m$ depending only on $d,\epsilon$ satisfying the following. 
Assume that 
\begin{itemize}
\item $X$ is an $\epsilon$-lc projective variety of dimension $d$,

\item $X\to Z$ is a contraction,

\item $X$ is of Fano type over $Z$, and 

\item $N$ is an integral divisor which is nef over $Z$.

\end{itemize}

Let $f\colon X\to V/Z$ be the contraction defined by $N$. Then there exists an integral divisor $L$ on $V$ 
such that 
\begin{itemize}
\item $mN\sim f^*L$, and  
\item $|mL|_G|$ defines a birational map where $G$ is the generic fibre of $V\to Z$.
\end{itemize}

\end{conj}

The latter basically says that $mL$ defines a birational map relatively over $Z$.
Note that since $X$ is of Fano type over $Z$ and $N$ is nef over $Z$, $N$ is semi-ample over $Z$ so 
indeed it defines a contraction. The global case, i.e. 
when $Z$ is a point, is a generalisation of Theorem \ref{t-eff-bir-nefbig-weil} for nef 
but not necessarily big divisors. 

The conjecture is stronger than it may look at first sight. For example, consider the case when $X$ is 
a $\Q$-factorial $\epsilon$-lc projective variety of dimension $d$ and $X\to Z$ is a Mori fibre space where 
$Z$ is a curve. Assuming $N$ is the reduction of a fibre of $X\to Z$ (that is, a fibre with induced reduced structure), 
the conjecture implies that $mN\sim 0/Z$ for some $m$ depending only on $d,\epsilon$. It is not hard to see 
that this implies Shokurov's conjecture on boundedness of singularities in fibrations 
[\ref{B-sing-fano-fib}, Conjecture 1.2].  
Conversely, Shokurov's conjecture combined with Theorem \ref{t-eff-bir-nefbig-weil} 
implies the above conjecture. If $N$ is big over $Z$, then we can apply Theorem \ref{t-eff-bir-nefbig-weil}. 
Assume $N$ is not big over $Z$. 
Running an MMP on $K_X$ over $V$ and replacing $X$ with the 
resulting model we can assume we have a Mori fibre space $h \colon X\to T/Z$ such that $N\sim_\Q 0/T$. 
The fibres of $X\to T$ over closed points are $\epsilon$-lc Fano varieties, so they belong to 
a bounded family [\ref{B-BAB}, Theorem 1.1]. Thus $lN$ is Cartier near the generic fibre of 
$X\to T$ for some bounded $l\in \N$, so $lN\sim 0$ over the generic point of $T$. 
Applying Shokurov's conjecture to $(X,B)\to T$ for some general 
$0\le B\sim_\Q -K_X/T$ and then applying Lemma \ref{l-descent-weil-divs-fib} and 
replacing $l$ we can assume $lN\sim h^* D$ for some integral divisor $D$. Replacing 
$X,N$ with $T,D$ we can apply induction on dimension.

\begin{rem}\label{rem-6}
\emph{
One may wonder if in Theorem \ref{t-eff-bir-nefbig-weil} and Conjecture \ref{conj-nef-integral-FT} (say when $Z$ is a point) 
we can choose $m$ so that $|mN|$ is base point free. But this is not the case. 
 For example, there exist  
a log smooth lc pair $(X,S)$ of dimension two and a nef and big divisor $N$ such that 
$K_X+S\sim 0$ and $S$ is the stable base locus of $N$ [\ref{Lazar}, \S 2.3.A] ($X$ can be obtained by blowing up 
$\PP^2$ in $12$ suitable points on an elliptic curve; $S$ is then the birational transform of 
the elliptic curve; $N$ is also constructed using the $12$ points). 
In particular, $N-K_X\sim N+S$ is big 
but $|mN|$ is not base point free for any $m$.}

\emph{
On the other hand, there are $3$-folds $X$ with terminal singularities and 
$N:=K_X$ ample but with arbitrarily large Cartier index, so $|mN|$ cannot be free for 
bounded $m$. 
}
\end{rem}


\vspace{2cm}

\small
\textsc{Yau Mathematical Sciences Center} \endgraf
\textsc{JingZhai Building, Tsinghua University} \endgraf
\textsc{ Hai Dian District, Beijing, China 100084  } \endgraf
\vspace{0.5cm}
\email{Email: birkar@tsinghua.edu.cn\\}

\end{document}